\documentclass[opre,nonblindrev]{informs3}

\DoubleSpacedXI 

\usepackage{endnotes}
\let\footnote=\endnote
\let\enotesize=\normalsize
\def\notesname{Endnotes}%
\def\makeenmark{\hbox to1.275em{\theenmark.\enskip\hss}}
\def\enoteformat{\rightskip0pt\leftskip0pt\parindent=1.275em
  \leavevmode\llap{\makeenmark}}
  
\newcommand*\diff{\mathop{}\!\mathrm{d}}

\def \E {\mathbb{E}}
\def \P {\mathbb{P}}
\def \N {\mathbb{N}}
\def \R {\mathbb{R}}

\def \-> {\to}

\def \B {\mathcal{B}}
\def \xb {\mathbf{x}}

\def \var {\operatorname{Var}}

\def \argmin {\operatorname{argmin}}

\def \Var {\mbox{\bf Var}}
\def \pil {\mathrm{P}}
\def \cop {\mathrm{C}}
\def \myo {\mathrm{M}}
\def \base {\mathrm{B}}
\def \pol {\mathrm{A}}
\def \cbs {\mathrm{R}}

\def \G {\mathcal{G}}

\usepackage{enumerate}

\usepackage[hidelinks]{hyperref}
\usepackage{pgfplots}
\usepackage{tabularx}
\usepackage{multicol}
\usepackage{multirow}
\usepackage{booktabs}
\usepackage{bigstrut}
\usepackage{subcaption}

\usepackage{comment}

\def \L {\mathcal{L}}

\usepackage{natbib}
 \bibpunct[, ]{(}{)}{,}{a}{}{,}%
 \def\bibfont{\small}%
 \def\bibsep{\smallskipamount}%
 \def\bibhang{24pt}%
\TheoremsNumberedThrough     
\ECRepeatTheorems

\EquationsNumberedThrough    

\MANUSCRIPTNO{not assigned} 

\begin{document}


\RUNAUTHOR{Van Jaarsveld, Arts}

\RUNTITLE{Projected Inventory Level Policies for Lost Sales Inventory Systems}

\TITLE{Projected Inventory Level Policies for Lost Sales Inventory Systems: Asymptotic Optimality in Two Regimes}

\ARTICLEAUTHORS{%
\AUTHOR{Willem van Jaarsveld}
\AFF{School of Industrial Engineering, Eindhoven University of Technology, Eindhoven, the Netherlands, PO BOX 513, 5600MB, \EMAIL{w.l.v.jaarsveld@tue.nl}}

\AUTHOR{Joachim Arts}
\AFF{Luxembourg Centre for Logistics and Supply Chain Management, Department of Economics and Management, University of Luxembourg, Luxembourg City, Luxembourg,  6, rue Richard Coudenhove-Kalergi L-1359 \EMAIL{joachim.arts@uni.lu}} 

}

\ABSTRACT{%
We consider the canonical periodic review lost sales inventory system with positive lead-times and stochastic i.i.d. demand under the average cost criterion.
We introduce a new policy that places orders such that the expected inventory level at the time of
arrival of an order is at a fixed level and call it the Projected Inventory Level (PIL) policy.
We prove that this policy has a cost-rate superior to the equivalent system where excess demand is back-ordered instead of lost and
is therefore asymptotically optimal as the cost of losing a sale approaches infinity under mild distributional assumptions.
We further show that this policy dominates the constant order policy
for any finite lead-time and is therefore asymptotically
optimal as the lead-time approaches infinity for the case of exponentially distributed demand per period.
Numerical results show this policy also performs superior relative to other policies.
}%




\maketitle

%


\section{Introduction}
\label{sec:Introduction}

The periodic review inventory system with lost sales, positive lead time and i.i.d. demand
is a canonical problem in inventory theory. The decision maker is interested in the
average cost-rate of this system. The optimal policy for such
a system can be computed in principle by stochastic dynamic programming, but
it is not practical due to the curse of dimensionality. 
Research has therefore focused on devising
heuristic policies for the lost sales inventory system. Although many variants of lost sales inventory systems exist, results for the canonical system are 
important as they serve as building blocks to design good policies for more
intricate lost sales inventory systems. \cite{BijvankVis2011} review the literature on many such more intricate inventory systems with lost sales.

There are two simple heuristic policies for the canonical lost sales system
that appeal to both practitioners and academics. These policies are the base-stock policy and the constant order policy. 

The base-stock policy places an order each period such that the inventory
position (inventory on-hand + outstanding orders) is raised
to a fixed {\em base-stock level}. This policy is prevalent in practice
due to its intuitive structure and because it is the optimal policy
when excess demand is not lost but back-ordered. The most important merit 
of the base-stock policy is that it is asymptotically optimal as the cost 
of a lost sale approaches infinity under mild conditions on the demand
distribution \citep{Huhetal2009MS}. This asymptotic optimality is robust in 
the sense that it holds for a broad class of heuristics to compute 
base-stock levels \citep{Bijvanketal2015}.

The constant order policy orders the same amount in each period regardless 
of the state of the system. Although this may seem naive at first, this 
policy is asymptotically optimal as the lead-time approaches infinity \citep{Goldbergetal2016},
and can outperform the base-stock policy for long lead-times and moderate 
costs for a lost sale. \cite{xin2016optimality} show that the constant order policy converges to
optimality exponentially fast in the lead time. 

The asymptotic optimality results of \cite{Huhetal2009MS} and
\cite{Goldbergetal2016} are both elegant and useful for
practice. We believe that these results should also inform the design of plausible heuristic
policies: With the knowledge that such asymptotic optimality results are 
attained by relatively simple policies, new
heuristics for the lost sales system should be designed
to be asymptotically optimal for long lead times
{\em and} large lost sales penalties. Unfortunately, the constant
order policy is not asymptotically optimal for large lost sales
penalties and the base-stock policy is not asymptotically
optimal for long lead times. This paper introduces a single parameter policy that is asymptotically optimal for large lost sales penalties 
under mild assumptions on the demand distribution and also for
long lead times when demand has an exponential distribution.
We call this policy the {\em projected inventory level} (PIL)
policy.

The PIL policy places 
orders such that the expected inventory level at the time of
arrival of an order is raised to a fixed level which
we call the projected inventory level. The PIL policy is intuitive for academics and practitioners alike. In fact, the base-stock policy for the canonical inventory 
system where excess demand is back-ordered, rather than lost, is
also a projected inventory level policy: Although the usual
interpretation of a base-stock policy in a system with
back-orders is that it raises the inventory position to a fixed
level, it is equivalent to say that it raises the expected
inventory level at the time of order arrival to a fixed level.
(These two policies are not equivalent in the lost sales inventory system.)
We exploit this equivalence and use it to compare the canonical inventory
systems with lost sales and back-orders respectively when all parameters are
identical. We prove that the cost-rate for the canonical lost sales system is lower
than the cost-rate for the canonical back-order system under the same
projected inventory level. As a corollary to this, we 
find that optimal cost-rate of the canonical lost
sales system is lower than the optimal cost-rate for the equivalent
canonical back-order system. This means we recover a main result of  \cite{janakiramanetal2007} via a new and different proof. The stochastic comparison technique used by \cite{janakiramanetal2007} holds for general convex per period cost functions but does not construct a policy. Our result is \emph{constructive} in the sense that we identify a specific policy (the PIL policy) for which the costs in the lost sales system are lower than the optimal costs for the back-order system. Our construction requires that the cost per period is linear in on-hand inventory and lost sales, which is the most commonly used per period cost function.

The PIL policy also mimics the behavior of the constant order policy for long lead-times. We make this notion rigorous when demand per period has an exponential distribution. In that case, we show that the projected inventory level policy can be interpreted as a one-step policy improvement on the (bias function of) the constant order policy; we believe this to be an interesting proof technique. Under the same assumption, we show that the projected inventory level policy dominates the constant order policy
for any finite lead-time $\tau$. The PIL policy therefore inherits the property of the constant order policy that the gap with the optimal policy decreases exponentially with the lead-time cf. \cite{xin2016optimality}.

Note that the projected inventory level policy has a single
parameter and yet uses all the information in the state vector without aggregating it into the inventory position. This is a feature shared by the Myopic policy where orders are placed to minimize the expected cost in the period that the order arrives. These policies all require a projection of the inventory level at the time of order arrival.
The myopic and PIL policy can therefore both be considered as ``projection'' policies.
Empirically projection policies perform exceptionally well (see \cite{Zipkin2008numeric} and Section \ref{sec:NumericalResults}), but there is no theoretical underpinning that explains why such policies perform so well empirically. In particular, there are no known asymptotic optimality results for such policies. This paper contributes asymptotic optimality results for the projected inventory level policy in two asymptotic regimes.

Policies that are asymptotically optimal for high lost sales costs and long lead times can be constructed by using multiple
parameters to make the policy behave as either a base-stock policy or a constant order policy when needed. For example, a policy may order a convex combination of the base-stock and constant order policy decision, or cap the order placed by a base-stock policy \cite{johansen2008,xin2019}. The projected inventory level policy is different in that the parameter can not be set such that it trivially reduces to either a base-stock or constant order policy. Asymptotic optimality proofs therefore rely on new ideas.

In summary, this paper makes the following contributions:
\begin{enumerate}
    \item We introduce the projected inventory level policy and show that it is a natural generalization of the base-stock policy to systems where sales are lost rather than back-ordered. Furthermore, it is a single parameter policy that utilizes all state information without aggregating it into the inventory position, i.e., it is a projection policy.
    \item We provide the first tractable policy for the canonical lost sales inventory system with better performance than the optimal policy of the equivalent canonical back-order system. The proof uses a comparison based on associated random variables.
    \item We prove that the projected inventory level policy is asymptotically optimal as the penalty for a lost sale approaches infinity under mild conditions on the demand distribution.
    \item We prove that the projected inventory level policy is a 1-step policy improvement upon the constant order policy and dominates the performance of the constant order policy when demand has an exponential distribution.
    \item We prove that the projected inventory level policy is asymptotically optimal as the lead time approaches infinity when demand has an exponential distribution. The proof approach is to show dominance of the PIL policy over the constant order policy such that it inherits is asymptotic optimality properties.
    \item We demonstrate numerically that the projected inventory level has superior performance also outside of the regimes where it is asymptotically optimal. We also provide results that aid with efficient computation within this class of policies.
\end{enumerate}

\section{Brief Literature Review}
\label{sec:LiteratureReview}

The canonical lost sales inventory system was first studied by \cite{KarlinScarf1958} and found
to have a more complicated optimal ordering policy than the canonical inventory system with
back-orders. The optimal ordering policy for the lost sales system depends on each outstanding order. Sensitivities of the optimal ordering decision to each outstanding order where first characterized by
\cite{Morton1969} and the analysis was later streamlined by \cite{Zipkin2008structure}.
Despite these results, computation and implementation of the optimal policy is not practical. Most of the literature studies heuristic policies without any optimality guarantees \citep[e.g.][]{Donselaaretal1996,BijvankJohansen2012,sunetal2014,vanJaarsveld2020} but with numerically favorable performance.  
Notable exceptions are \cite{Levietal2008MOOR} who prove that their dual balancing policy has an optimality gap of at most 100\% and \cite{chenetal2014PPTAS} who provide a pseudo polynomial time approximation scheme. Another active area of research is the study of base-stock policies when the demand distribution is unknown and must be learned online \citep{Huhetal2009MOOR,Zhangetal2019MS,agrawal2019learning}. We refer to \cite{BijvankVis2011} for a general review of lost sales inventory models.

We use the remainder of this brief literature review to focus on asymptotic optimality results as this is the focus of this paper. The notion of asymptotic optimality in lost sales and other difficult inventory systems has gained recent traction. \cite{goldbergetal2019} provide a survey of such results and outline methodologies to prove such results. The most important results for the lost sales inventory system are asymptotic optimality of base-stock policies as the cost of losing a sale approaches infinity \citep{Huhetal2009MS,Bijvanketal2015} and asymptotic optimality of constant order policies as the lead-time approaches infinity \citep{Goldbergetal2016,xin2016optimality,buetal2020,xin2019,Baietal2023}.
A natural question is whether any intuitive policies exist that are asymptotically optimal in both regimes. \cite{xin2019} propose the capped base-stock policy which was first introduced by \cite{johansen2008}. Under this policy orders are placed to reach a base-stock level except when this would cause the order size to exceed the cap. The parameters of this policy can be set such that it effectively reduces to either a constant order or base-stock policy. \cite{xin2019} show that such a policy is asymptotically optimal for long lead-times. It is straightforward to show that a capped base-stock policy is also asymptotically optimal in the other regime.
The PIL policy has a single parameter that cannot be set such that it effectively reduces to either a constant order or base-stock policy. Despite this, we provide asymptotic optimality results both for long lead-times and high per unit lost sales costs.

Our analysis is based on comparison against simpler policies for which asymptotic optimality results have already been established. The comparison uses association of random variables when comparing the PIL with the base-stock policy. A similar technique has been used to bound the order fill-rates in assemble-to-order systems by \cite{song1998}. We use policy improvement to compare the PIL policy to the constant order policy. This idea is often used to create an improved policy for a system that suffers from the curse of dimensionality so that only a simple policy can be analyzed \citep[see e.g.][]{tijms2003,haijemavanderwal2008}. We are not aware of prior work that uses association of random variables and/or policy improvement to establish asymptotic optimality results.

\section{Model}
\label{sec:model}
We consider a periodic review lost sales inventory system. In each period $t\in \N_0$ (where $\N_0:=\N\cup\{0\}$) we place an order that will arrive after a lead time of $\tau\in \N_0$ periods, i.e. at the start of period $t+\tau$. The inventory level at the beginning of period $t$, after receiving the order that was placed in period $t-\tau$, is denoted $I_t$. We denote the inventory level at the end of period $t$ as $J_t$. The order placed in period $t$ is denoted $q_{t+\tau}$ such that the order that arrives in period $t$ is denoted $q_t$. Demand in period $t$ is denoted by $D_t$ and $\{D_t\}_{t=0}^\infty$ is an i.i.d. sequence of random variables with $\mu:=\E[D_t]<\infty$. We let $D$ denote the generic one period demand random variable and its distribution  $F(x):=P(D\le x)$.  

Demand is satisfied from inventory whenever possible. If demand in a period exceeds the inventory level, there will be lost sales denoted by $L_t$ in period $t$. The system dynamics are given as:
\begin{align}
L_t & = (D_t-I_t)^+ \label{eq:dynamics1}\\
J_t & = (I_t-D_t)^+ \label{eq:dynamics2}\\
I_t & = (I_{t-1} - D_{t-1})^+ + q_{t} = J_{t-1}+q_t,\label{eq:dynamics3}
\end{align}
where $(x)^+:=\max(0,x)$. The state of this system in period $t$ is given by $\xb_t=(I_t,q_{t+1},q_{t+2},\ldots,q_{t+\tau-1})\in \R_+^{\tau}$. Since we focus on long run costs, and for ease of exposition, we assume the initial state $\xb_0$ to be $\mathbf{0}$, i.e. $I_0=0$ and $q_t=0$ for $t<\tau$ unless stated otherwise \citep[cf.][]{xin2016optimality}. For convenience in notation we define  $D[a,b]:=\sum_{i=a}^b D_i$ and similarly define $I[a,b]$, $J[a,b]$, $q[a,b]$ and $L[a,b]$.

In each period $t$ we may decide $q_{t+\tau}\ge 0$ based on $\xb_t$ and $t$. For our purposes, it will be convenient and sufficient to represent policies by a set of functions $\pol=\{\pol_t,t\in \N_0\}$, where $\pol_t$ maps states $\xb_t\in \R_+^{\tau}$ to actions $q_{t+\tau}\ge 0$. 

We consider three policies in the analytical sections of this paper. The base-stock policy $\base^S$ \citep[cf.][]{Huhetal2009MS} and the constant order policy $\cop^r$ \citep[cf.][]{xin2016optimality} are defined as follows:
 \begin{align*}
\base^S_t(\xb_t) & := \big(S-I_t-q[t+1,t+\tau-1]\big)^+ \\
\cop^r_t(\xb_t) &:= r.
\end{align*}
Here, $S$ and $r$ are the base-stock level and the constant order quantity, respectively. We assume for stability that $r<\mu$ \citep[cf.][]{xin2016optimality}. The projected inventory level policy $\pil^U$ with projected inventory level $U$ is given by 
\begin{align}
\label{eq:PUt}
\pil^U_t(\xb_t) & := (U-\E[J_{t+\tau-1}|\xb_t])^+ = \left(U-I_t -q[t+1,t+\tau-1] -\E\big[L[t,t+\tau-1]\big| \xb_t \big] +\tau\mu \right)^+.
\end{align}
Note that the expectations in \eqref{eq:PUt} take into account the state at time $t$. Therefore, at time $t$ the PIL policy places an order to raise the expected inventory level at time $t+\tau$ to $U$, if possible. 

The following lemma specifies conditions that ensure that the projected inventory level $U$ can be attained in every period; the proof is in Appendix \ref{sec:ProofAttainment}. 
\begin{lemma}
\label{lemma:Attainment}
For any given PIL policy, it is possible to place a non-negative order in each period $t\geq0$  to attain the projected inventory level $U\geq0$ provided that it is possible to do so in period 0 (i.e. provided $\xb_0$ satisfies $\E[J_{\tau-1}|\xb_0]\leq U$). 
\end{lemma}

Any excess inventory at the end of a period incurs a holding cost $h>0$ per item per period. Any lost sales accrued during a period incur a lost sales penalty cost $p>0$ per item lost. Denote the costs incurred in period $t$ by $c_t:=h J_t + p L_t$, and let $c[a,b]:=\sum_{t=a}^b c_t$. We write $c[a,b](\pol)$ to make the dependence on the policy $\pol$ explicit. The cost-rate associated with a policy $\pol$ is then
\[
C(\pol):=\limsup_{T\to\infty}\E\left[ \frac{1}{T-\tau+1}  c[\tau,T](\pol)\right]
\]
Let $S^*\in \argmin_S{C(\base^S)}$, $r^*\in \argmin_r{C(\cop^r)}$, and $U^*\in\argmin_U{C(\pil^U)}$ denote the optimal base-stock level, constant order quantity, and projected inventory level respectively. Let $C^*$ denote the long run expected costs of an optimal policy.

\section{Long lead time asymptotics}
\label{sec:LongLeadTimes}
Constant-order policies were proven to be asymptotically optimal for long lead-time by \cite{Goldbergetal2016} \citep[see also][]{xin2016optimality}. This result has deepened the understanding of lost sales inventory systems. Empirically, we observe that the PIL policy outperforms the constant order policy, also for long lead-times. In this sense, the PIL policy is unlike the base-stock policy, which cannot match the constant order policy performance for large lead-times. In this section, we will theoretically underpin this finding. In particular, we prove that the projected inventory policy is in expectation superior to the constant order policy when demand is exponentially distributed. 

Our analysis is based on the following simple idea.
Consider the total costs incurred from time $t+\tau$ up to time $T$, given the state $\xb_t$ and order $q$ and assuming $q_k=r$ for $k>t+\tau$, and denote this total cost by $F^T(q) = hJ[t+\tau,T]+pL[t+\tau,T]$. Define $f(q|\xb_t)=\lim_{T\to\infty}\E[ F^T(q)-F^T(r) \mid \xb_t]$. A sensible policy may decide $q_{t+\tau}\in\argmin_{q\geq 0} f(q|\xb_t)$ for any pipeline $\xb_t$. This policy may be recognized as a single-step policy improvement to the constant order policy, and it may therefore be expected to dominate the constant order policy with order quantity $r$. 

In the following, we sketch a heuristic argument that shows
$\pil^U (\xb_t)\in \argmin_{q\geq0} f(q|\xb_t)$ with $U=p(\mu-r)/h$. This heuristic argument as well as the intuition that this policy dominates the constant order policy will be made rigorous in Section \ref{sec:LongLeadTimeRigorous}.

Let us determine the $q=q_{t+\tau}$ that minimizes $f$ for some initial state $\xb$. Increasing $q$ by $\epsilon$ has two effects on the infinite horizon costs: 1) $\epsilon$ more demand is eventually satisfied (not lost) and a penalty cost $p\epsilon$ is averted 2) From time $t+\tau$ until the first stockout, the inventory level increases by $\epsilon$ (for $\epsilon$ \emph{small}). Suppose $t+\tau=0$ for notational convenience. Let $R$ denote the time of the first lost sale from period $0$: $R=\min\{t\in\N_0|L_t>0\}$. Then
\begin{align}
\frac{df(q)}{dq}=\lim_{\epsilon\to 0} \frac{f(q+\epsilon)-f(q)}{\epsilon}=
\lim_{\epsilon\to 0} \frac{-p\epsilon +h\E[R]\epsilon}{\epsilon}=h\E[R]-p.\label{eq:fderivative}
\end{align}
The expectation of $R$ can be found by noting that $R$ is a stopping time and $\E[L_R]=\mu$ due to the memoryless property of the exponential demand distribution. For $t<R$ there are no stockouts so that $J_t=J_{t-1}+r-D_t$. Thus $J_t=I_0-D[0,t] + t r =J_{-1}+q -D[0,t] + t r$ while for $t=R$ we have $J_R=0$ and $L_R=-J_{-1}-q + D[0,R] - R r$. In particular this implies (using Wald's identity) $\E[L_R]=-\E[J_{-1}]-q+(\E[R]+1)\mu-\E[R]r$.  Using that also $\E[L_R]=\mu$ and solving for $\E[R]$ yields
$\E[R]=\frac{\E[J_{-1}]+q}{\mu-r}$. Thus by \eqref{eq:fderivative}, $f$ must be a parabola in $q$ and $df(q)/dq=0$ if and only if $q=\frac{p(\mu-r)}{h}-\E[J_{-1}]$. Since $I_0=J_{-1}+q$, this holds if and only if $q$ follows a projected inventory level policy with level $U=\frac{p(\mu-r)}{h}$.

\subsection{Dominance of PIL policies over COP policies}
\label{sec:LongLeadTimeRigorous}

The heuristic argument above can be made rigorous as follows. 
Observe that the orders under $\cop^r$ are independent of the state. Hence, the \emph{bias} (or relative value function) of the $\cop^r$ policy can be expressed as a function of inventory level only as in the following definition. 

\begin{definition}\label{def:bias}
Let $r\in [0,\E[D])$, and let $g^r=C(\cop^r)$ be the long run average costs of the $\cop^r$ policy. Then the bias $\mathcal{H}^r(\cdot):\R^+\to\R$ associated with $\cop^r$ satisfies:
\begin{align}
 \mathcal{H}^r(x) = \E_{D}\big[ h(x-D)^+ + p (D-x)^+ + \mathcal{H}^r((x-D)^++r)\big]- g^r,\label{eq:biasdef}
\end{align}
for any non-negative $x\geq 0$. To make the bias unique we also impose $\mathcal{H}^r(0)=0$. 
\end{definition}
This bias $\mathcal{H}^r(x)$ can be interpreted as the additional cost
over an infinite horizon of starting with $x$ items in inventory and a pipeline with orders of size $r$ instead of starting with 0 items in inventory and a pipeline with orders of size $r$. Intuitively, $f(q|\xb_t)$ equals $\E[\mathcal{H}^r(J_{t+\tau-1}+q)-\mathcal{H}^r(J_{t+\tau-1}+r)|\xb_t]$, and informed by the heuristic argument made earlier one can guess that, like $f$, this bias should be a parabola. 
\begin{lemma}
\label{lem:bias}
When demand has an exponential distribution, the bias of the constant order policy with constant order quantity $r<\mu$ is a parabola:
\begin{equation}
\label{eq:biasexplicit}
    \mathcal{H}^r(x)=\frac{h}{2(\mu-r)}x^2-px,
\end{equation}
and $g^r=C(\cop^r)=p(\mu-r) + h \frac{r^2}{2(\mu-r)}$.
\end{lemma}
The proof of Lemma \ref{lem:bias} can be found in Appendix \ref{sec:ProofBias} and is a straightforward verification that the proposed solution satisfies the definition.

Confirming our intuition, we next show that the policy found by a single improvement step on the bias of a constant order policy is a projected inventory level policy.
\begin{lemma}\label{lem:pilisimproving}
Let $D$ be exponentially distributed and for any $h,p, r<\mu$, let $U(r)=p(\mu-r)/h$. For any $\xb_t$, if $q_{t+\tau}=\pil^{U(r)}(\xb_t)$ then $q_{t+\tau}\in \argmin_{q\ge 0} \E[\mathcal{H}^r(J_{t+\tau-1}+q)\mid\xb_t]$.
\end{lemma}
\begin{proof}{Proof.}
 Using \eqref{eq:biasexplicit}, one may derive $\mathcal{H}^{r}(x)=a_1(x-U(r))^2+a_2$ with $a_1=\frac{h}{2(\mu-r)}>0$ and $a_2=\frac{-p^2(\mu-r)}{2h}$, thus $U(r)$ is the unique minimizer of $\mathcal{H}^{r}(\cdot)$. Now observe that
\begin{align}
 \E[\mathcal{H}^{r}(J_{t+\tau-1}+q)\mid\xb_t] &= \E[a_1(J_{t+\tau-1}+q-U(r))^2+a_2\mid \xb_t] \notag\\
&=a_1\left[\var[J_{t+\tau-1}+q-U(r)\mid\xb_t] + \E[J_{t+\tau-1}+q-U(r)\mid\xb_t]^2 \right]+a_2 \notag\\
\label{eq:EHq}
&=a_1\var[J_{t+\tau-1}\mid\xb_t] + a_1(\E[J_{t+\tau-1}\mid\xb_t]+q-U(r))^2 +a_2,
\end{align}
where the final equality follows because $\var[q \mid \xb_t]=0$ for any deterministic policy $\pol$. Clearly $q=\pil^{U(r)}(\xb_t)=(U(r)-\E[J_{t+\tau-1}\mid \xb_t ])^+$ minimizes \eqref{eq:EHq}.\Halmos
\end{proof}

For our continuous state-space, continuous action-space model, there appear to be no standard Markov decision process results that can be leveraged to prove from Lemma~\ref{lem:pilisimproving} that the PIL dominates the constant order policy. Our proof uses the following result, the proof of which appears in Appendix~\ref{sec:ProofMartingaleexpression}:
\begin{lemma}\label{lem:Martingaleexpression}
Let $t_1\leq t_2$, $t_1,t_2\in\N_0$, $r\in [0,\E[D])$, $g^r=C(\cop^r)$, and suppose $q_t=r$
for all $t\in \{t_1+1,\ldots,t_2\}$. Then
\[ \E_{D_{t_1},\ldots,D_{t_2}}\left[ c[t_1,t_2](\cop^r) \mid I_{t_1}\right] = \mathcal{H}^r(I_{t_1})-\E_{D_{t_1},\ldots,D_{t_2}}[\mathcal{H}^r(I_{t_2+1})|I_{t_1}] + (t_2+1-t_1)g^r. \]
\end{lemma} 

We are now ready to establish the main result of this section. \begin{theorem}\label{thm:PilOutperformsCOP}
If demand has an exponential distribution then the best PIL policy $\pil^{U^*}$ outperforms the best constant order policy $\cop^{r^*}$. In particular
$C(\pil^{U^*})\leq  C(\pil^{U(r^*)})\le C(\cop^{r^*})$ for any $\tau\in\N_0$, where $U(r)$ is given by Lemma~\ref{lem:pilisimproving}. 
\end{theorem}
\begin{proof}{Proof.}
We focus on bounding $\E[c[\tau,T](\pil^{U(r^*)}) - c[\tau,T](C^{r^*})]$. A device in the proof will be a policy $\G^{\tilde{t}}$ that places the first $\tilde{t}\in \N_0$ orders using the projected inventory policy $\pil^{U(r^*)}$, and subsequent orders using the optimal constant-order policy $\cop^{r^*}$:
\[ \G^{\tilde{t}}_t(\xb):=\begin{cases}
\pil^{U(r^*)}_t(\xb), & t< \tilde{t}\\
\cop^{r^*}_t(\xb)=r^*, & t\ge \tilde{t}.
\end{cases}
 \]
Let $I_{t}(A)$ denote the random variable $I_{t}$ when policy $\pol$ is adopted, and let $\bar{t}=\tilde{t}+\tau$. We will compare expected interval costs for the policies $\G^{\tilde{t}+1}$ and $\G^{\tilde{t}}$:
\begin{align}
\E\left[c[\tau,T](\G^{\tilde{t}+1})-c[\tau,T](\G^{\tilde{t}})\right]&=\E\left[c[\bar{t},T](\G^{\tilde{t}+1})-c[\bar{t},T](\G^{\tilde{t}})\right]\nonumber\\
&=\E\left[\E_{D_{\bar{t}},\ldots,D_{T}}\left[ c[\bar{t},T] \Big| I_{\bar{t}}(\G^{\tilde{t}+1})\right] -\E_{D_{\bar{t}},\ldots,D_{T}}\left[ c[\bar{t},T] \Big| I_{\bar{t}}(\G^{\tilde{t}})\right]    \right]\nonumber\\
&=\E[\mathcal{H}(I_{\bar{t}}(\G^{\tilde{t}+1}))-\mathcal{H}(I_{\bar{t}}(\G^{\tilde{t}}))-\mathcal{H}(I_{T+1}(\G^{\tilde{t}+1}))+\mathcal{H}(I_{T+1}(\G^{\tilde{t}}))].\label{eq:costdeltaidentity}
\end{align}
Here and elsewhere in this proof, $\mathcal{H}$ denotes $\mathcal{H}^{r^*}$. Also, the first equality follows because $\G^{\tilde{t}+1}$ and $\G^{\tilde{t}}$ coincide for $t< \tilde{t}$, and thus, since $\xb_0=\mathbf{0}$ for $t\le \tilde{t}$ the distributions of $\xb_{t}$, $J_{t+\tau-1}$, $L_{t+\tau-1}$ and $c_{t+\tau-1}$ are the same for the two policies. The second equality is by conditioning on the inventory at time $\bar{t}$. For the third equality, note that $\G^{\tilde{t}+1}_t=\G^{\tilde{t}}_t=\cop^{r^*}$ for $t\ge \tilde{t}+1$, hence $q_{t}=r^*$ for $t\ge \tilde{t}+1+\tau$ for both policies, and hence, we can substitute the identity of Lemma~\ref{lem:Martingaleexpression}. 

Now note that $I_{\bar{t}}(\G^{\tilde{t}+1}) = J_{\bar{t}-1}+\pil_{\tilde{t}}^{U(r^*)}(\xb_{\tilde{t}})$ and $I_{\bar{t}}(\G^{\tilde{t}}) = J_{\bar{t}-1}+r^*$, while $\xb_{\tilde{t}}$ and $J_{\bar{t}-1}$ are identically distributed for both policies since they coincide for $t<\tilde{t}$. We condition on $\xb_{\tilde{t}}$ and find:
\begin{align}
\E\left[\mathcal{H}\left(I_{\bar{t}}(\G^{\tilde{t}+1})\right)-\mathcal{H}\left(I_{\bar{t}}(\G^{\tilde{t}})\right)\right]&=\E\left[\E\left[\mathcal{H}\left(I_{\bar{t}}\left(\G^{\tilde{t}+1}\right)\right)-\mathcal{H}\left(I_{\bar{t}}\left(\G^{\tilde{t}}\right)\right)\big|\xb_{\tilde{t}}\right]\right]\nonumber\\
&=\E\left[\E\left[\mathcal{H}\left(J_{\bar{t}-1}+\pil_{\tilde{t}}^{U(r^*)}\left(\xb_{\tilde{t}}\right)\right)-\mathcal{H}(J_{\bar{t}-1}+r^*)|\xb_{\tilde{t}}\right]\right]\nonumber\\
&=\E\left[\min_{q\ge 0}\left(\E\left[\mathcal{H}\left(J_{\bar{t}-1}+q\right)-\mathcal{H}\left(J_{\bar{t}-1}+r^*\right)|\xb_{\tilde{t}}\right]\right)\right]\nonumber\\
&=\E\left[\min_{q\in \R}\left(\E\left[\mathcal{H}\left(J_{\bar{t}-1}+q\right)-\mathcal{H}\left(J_{\bar{t}-1}+r^*\right)|\xb_{\tilde{t}}\right]\right)\right]\nonumber\\
&=-\E\left[\E\left[ a_1 \left(r^*-\pil_{\tilde{t}}^{U(r^*)}\left(\xb_{\tilde{t}}\right)\right)^2\middle|\xb_{\tilde{t}}\right]\right]\nonumber\\
& =-a_1\E\left[  \left(r^*-\pil_{\tilde{t}}^{U(r^*)}(\xb_{\tilde{t}})\right)^2\right] \label{eq:boundingthecostdifference}
\end{align}
For the third equality, we use Lemma~\ref{lem:pilisimproving}. For the fourth equality, observe that $\xb_0=0$, and hence  $\E[J_{\bar{t}-1}|\xb_{\tilde{t}}]\le U(r^*)$  (cf. Lemma~\ref{lemma:Attainment}), which implies that the minimum over $q\in \R$ is attained by an element $q\ge0$. 
For the fifth equality, we substitute equation \eqref{eq:EHq} and cancel terms. 

With this, we obtain:
\begin{align}
 \E\big[c[\tau,T](\pil^{U(r^*)})-c[\tau,T](\cop^{r^*})\big]&= \E\big[c[\tau,T](\G^{T+1-\tau})-c[\tau,T](\G^0)\big] \nonumber\\ 
&=\sum_{\tilde{t}=0}^{T-\tau} \E\big[c[\tau,T](\G^{\tilde{t}+1})-c[\tau,T](\G^{\tilde{t}})\big]\nonumber\\
&=\sum_{\tilde{t}=0}^{T-\tau}\E[\mathcal{H}(I_{\bar{t}}(\G^{\tilde{t}+1}))-\mathcal{H}(I_{\bar{t}}(\G^{\tilde{t}}))-\mathcal{H}(I_{T+1}(\G^{\tilde{t}+1}))+\mathcal{H}(I_{T+1}(\G^{\tilde{t}}))]\nonumber\\
&= \E[I_{T+1}(\G^0)] -\E[I_{T+1}(\G^{T+1-\tau})] +\sum_{\tilde{t}=0}^{T-\tau} \E\left[\mathcal{H}\left(I_{\bar{t}}(\G^{\tilde{t}+1})\right)-\mathcal{H}\left(I_{\bar{t}}(\G^{\tilde{t}})\right)\right] \nonumber\\
&= \E[I_{T+1}(\cop^{r^*})]-\E[I_{T+1}(\pil^{U(r^*)})]-a_1 \sum_{\tilde{t}=0}^{T-\tau} \E\left[  \left(r^*-\pil_{\tilde{t}}^{U(r^*)}(\xb_{\tilde{t}})\right)^2\right] \label{eq:finitehorizoncostdifference}
\end{align}
Here, the first equality holds by definition of $\G^{\tilde{t}}$. The second equality follows by expressing the difference as a telescoping sum. The third equality uses \eqref{eq:costdeltaidentity}. For the fourth equality, we rearrange and cancel terms. The final equality holds by definition of $\G^{\tilde{t}}$, and by \eqref{eq:boundingthecostdifference}. 

Using the definition of the cost-rate, we find 
\begin{align*}
C(\cop^{r^*})&= \lim_{T\to \infty} \E\left[\frac{1}{T-\tau+1}c[\tau,T](\cop^{r^*}) \right]\\
&=\lim_{T\to \infty}\E \left[\frac{1}{T-\tau+1}\left(c[\tau,T](\pil^{U(r^*)}) +a_1 \sum_{\tilde{t}=0}^{T-\tau}   \left(r^*-\pil_{\tilde{t}}^{U(r^*)}(\xb_{\tilde{t}})\right)^2\right) \right]\\
&=\limsup_{T\to \infty}\E \left[\frac{1}{T-\tau+1}c[\tau,T](\pil^{U(r^*)})  \right]+
a_1\liminf_{T\to \infty} \E\left[\frac{1}{T-\tau+1}\sum_{\tilde{t}=0}^{T-\tau}   \left(r^*-\pil_{\tilde{t}}^{U(r^*)}(\xb_{\tilde{t}})\right)^2\right]
\end{align*}
The first equality uses that for the constant order policy, the sequence converges such that the limit superior equals the limit. For the second equality, we use \eqref{eq:finitehorizoncostdifference} and $\lim_{T\to \infty}\frac{1}{T-\tau+1} (\E[I_{T+1}(\cop^{r^*})]-\E[I_{T+1}(\pil^{U(r^*)})])=0$. That this limit goes to zero follows because the steady state inventory level under the constant order policy exists (with finite mean), and because the inventory level under the PIL policy is bounded from below by $0$ and from above by $U(r^*)+\tau\mu$. This upper bound holds because $\pil^{U(r^*)}_t(\xb_t) \le \left(U(r^*)+\tau\mu-I_t -q[t+1,t+\tau-1]  \right)^+$, i.e. the inventory position after placing an order $I_t+q[t+1,t+\tau-1]+q_{\tau}$ cannot rise above $U(r^*)+\tau\mu$, (and it is below this bound initially since $\xb_0=\mathbf{0}$ by assumption). Thus the inventory level is bounded by $U(r^*)+\tau\mu$.

The third equality now follows because all limit points of the sequence $\left(\E\left[\frac{1}{T-\tau+1}\sum_{\tilde{t}=0}^{T-\tau}   \left(r^*-\pil_{\tilde{t}}^{U(r^*)}(\xb_{\tilde{t}})\right)^2\right]\right)_{T= \tau,\tau+1,\ldots}$ must be finite since $\pil_{\tilde{t}}^{U(r^*)}(\xb_{\tilde{t}})$ is bounded below by $0$ and above by $U(r^*)+\tau\mu$. Now note that for each limit point $x\in\R$ of that sequence, there must exist a subsequence that converges to that limit point. By the second equality above, the corresponding subsequence of $\left(\E \left[\frac{1}{T-\tau+1}c[\tau,T](\pil^{U(r^*)})\right]\right)_{T= \tau,\tau+1,\ldots}$ must converge to $y=C(\cop^{r^*})-x$. Thus every limit point $x$ of the former sequence must correspond to a limit point $y=C(\cop^{r^*})-x$ of the latter sequence, and vice versa. Also, the smallest limit point of the latter sequence must correspond to the largest limit point of the former sequence, which implies the third equality.

We thus find:
\begin{equation}
\label{eq:QuadDivergence}
C(\cop^{r^*})-C(\pil^{U(r^*)}) = a_1\liminf_{T\to \infty} \frac{1}{T-\tau+1} \sum_{\tilde{t}=0}^{T-\tau} \E\left[  \left(r^*-\pil_{\tilde{t}}^{U(r^*)}(\xb_{\tilde{t}})\right)^2\right]\ge 0. 
\end{equation}
This completes the proof.
 \Halmos
\end{proof}
It is noteworthy that the difference between the cost of a PIL policy and a constant order policy can be expressed as a function of the quadratic differences between the order decisions of both policies; see \eqref{eq:QuadDivergence}. The amounts by which the decisions of the PIL policy differ from a constant order policy also express how much better it performs.

\subsection{Asymptotic optimality as $\tau\to\infty$}  

With the results of \cite{xin2016optimality}, Theorem~\ref{thm:PilOutperformsCOP} establishes asymptotic optimality of the PIL policy as $\tau$ grows large:
\begin{theorem}
The PIL policy is asymptotically optimal for long lead-times when demand has an exponential distribution:
\[
\lim_{\tau\to\infty}\left(C(\pil^{U(r^*)})-C^{*}\right) = 0.
\]
\end{theorem}
This result follows directly from our Theorem~\ref{thm:PilOutperformsCOP}, and Theorem~1 in \cite{xin2016optimality}. In fact, it follows from these theorems that the optimality gap of PIL policy decays exponentially in $\tau$. 

\section{Penalty cost asymptotics}\label{sec:penaltycostsasymptotics}
The performance of the PIL-policy in the asymptotic regime that $p\to\infty$ will be studied
by using bounds in terms of the related inventory system in which demand in excess of inventory
is back-ordered rather than lost. Therefore, in this section, we first make a comparison to
the related canonical back-order system in Section \ref{sec:Comparison} and show that the cost of a PIL policy
for the lost sales system has lower cost than the optimal policy for the canonical back-order system. Then we provide our main result that the PIL policy is asymptotically optimal as the cost of losing a sale approaches infinity in Section \ref{sec:optimalityLargeP}.

\subsection{Comparison to back-order system}
\label{sec:Comparison}

The back-order system with lead-time $\tau\in\N_0$, holding cost
parameter $h>0$ per period per item, and back-order cost $p>0$ per period per item is much better
understood than the same system with lost sales. Comparison between these two systems has been studied before
by e.g. \cite{janakiramanetal2007},
\cite{Huhetal2009MS}, and \cite{Bijvanketal2015} and we
mostly follow their notational conventions. 
The dynamics for the back-order system are:
\begin{align}
I_{t}^\B=I_{t-1}^\B-D_{t-1}+q^\B_{t}, \quad J^\mathcal{B}_t=(I_t^\mathcal{B}-D_t)^+, \quad
B_{t}=(D_t-I_t^\B)^+,
\end{align}
where $B_t$ and $I_t^\B$ denote respectively the number of items on back-order and the inventory level in period $t$. We use the superscript $\B$ to denote that quantities belong to the back-order system
(rather than the lost sales system) and generally denote this system as $\B$.
It is well known that the optimal policy for $\B$ is a base-stock policy. Under this policy the order in each period $t$
is placed to raise the inventory position to a fixed base-stock level $S$:
\begin{equation}
q^\B_{t+\tau}=S-I^\B_t-q^\B[t+1,t+\tau-1],
\end{equation}
and the optimal base-stock level for system $\B$ is given by the newsvendor equation:
\begin{equation}
S^* = \inf\left\{S: \P(D[0,\tau]\leq S)\geq \frac{p}{p+h} \right\}.
\end{equation}
The optimal average cost-rate for this system satisfies:
\begin{equation}
C^{\B *}=h\E\left[(S^*-D[0,\tau])^+\right]+p\E\left[(D[0,\tau]-S^*)^+\right].
\end{equation}
For the lost sales system we similarly define $c^\mathcal{B}_t:=hJ^\mathcal{B}_t+p B_t$ and
$c^{\mathcal{B}}[a,b]:=\sum_{t=a}^b c^{\mathcal{B}}_t$. We also write $c^\mathcal{B}[a,b](\pol)$ to make the dependence on the control policy $\pol$ explicit. With this notation we can express the cost-rate of a policy $\pol$ in system $\mathcal{B}$ as $C^\mathcal{B}(\pol)=\limsup_{T\to\infty}\E\left[\frac{1}{T-\tau+1} c^{\mathcal{B}}[\tau,T](\pol)\right]$.
The lost sales system described in Section \ref{sec:model} is denoted by $\L$, and the optimal cost for this
system is still denoted by $C^{*}$. 
Note that both systems are defined on the same probability space induced by the initial state and 
demand sequence. The main result of \cite{janakiramanetal2007} is that $C^{*}\leq C^{\B*}$, which is established via an ingenious stochastic comparison technique. The main result of this section is
that the best PIL-policy for $\L$ achieves lower cost than the optimal policy for $\B$: $C(\pil^{U^*})\leq C^{\B*}$. Via $C^{*}\leq C(\pil^{U^*})\leq C^{\B*}$, this result constitutes the first \emph{constructive} proof of the main result of \cite{janakiramanetal2007}: Unlike their stochastic comparison proof, we identify a specific (PIL) policy that yields a cost-rate for system $\L$ that is lower than the optimal cost-rate for system $\B$. In addition to this, this result also enables us to leverage results in \cite{Huhetal2009MS} to show (under mild conditions) that the PIL-policy is asymptotically optimal as $p$ grows large. 

The main idea behind the proof of this result is that the base-stock policy with base-stock level $S>\tau\mu$ in $\B$ is
also a PIL-policy with projected inventory $S-\tau\mu$ in system $\B$. Indeed observe that
\[
I^\B_{t+\tau}=I^\B_t+q[t+1,t+\tau]-D[t,t+\tau-1]=S-D[t,t+\tau-1],
\]
so that $\E[I^\B_{t+\tau}]=S-\tau\mu$. From this it immediately follows that $C^\mathcal{B}\left(\pil^{S^*-\tau\mu}\right)=C^{\mathcal{B}}\left(\base^{S^*}\right)=C^\mathcal{B^*}$ when
$S^*\geq\tau\mu$.

We will first show that $\E[L_t]\leq \E[B_t]$
for any period $t\geq \tau$ when $\mathcal{L}$ and $\mathcal{B}$ operate under the same PIL-policy with level $U\geq 0$. From that, we conclude that the cost-rate of $\mathcal{L}$ under the optimal
PIL policy is smaller than the optimal cost-rate for system $\mathcal{B}$.

The following technical lemma is needed in subsequent results. Its proof is in Appendix \ref{sec:ivosproof}.
\begin{lemma}
\label{lem:IvosLemma}
Let $X$ and $Y$ be random variables with joint distribution function $F(x,y)$, $-\infty<x,y<\infty$.
Then $\E[(X+Y)^+]=\int_{-\infty}^\infty \P(X\geq z, Y\geq -z) dz$.
\end{lemma}

Define the random variables 
\[
Y= L[0,\tau-1] - \E[L[0,\tau-1]], \qquad\mbox{and}\qquad X=D_{\tau}-I_{\tau}=D[0,\tau]-S-Y,
\]
for $S\geq\tau\mu$.
Observe that $X^+=L_{\tau}$ and that $X+Y=D[0,\tau]-S$ such that $(X+Y)^+=B_{\tau}$ 
under a PIL-policy with level $U=S-\tau\mu\geq0$.
Our aim will be to prove that $\E[L_{\tau}]\leq \E[B_{\tau}]$. To this end, 
we first prove that $X$ and $Y$ are \emph{associated} random variables (cf. \cite{Esaryetal1967}).

\begin{definition}
The random variables $(A_1,\ldots,A_n)=\mathbf{A}$ are said to be associated if
\[
\mbox{Cov}[f(\mathbf{A}),g(\mathbf{A})]\geq 0
\]
for all non-decreasing functions $f,g:\R^n \to\R$ for which the covariance above exists.
\end{definition}

\begin{lemma}
$X$ and $Y$ are associated random variables.
\end{lemma}

\begin{proof}{Proof.}
The random variables $\mathbf{D}=(D_0,\ldots,D_{\tau})$ are associated by Theorem 2.1 of \cite{Esaryetal1967}.
By property $P_4$ of \cite{Esaryetal1967} it suffices to show that $Y=f(\mathbf{D})$ and $X=g(\mathbf{D})$ are
non-decreasing functions (element wise).

We derive an expression for $L[0,t]$: The cumulative demands lost until $t$. Each arriving demand $D[0,t]$ until $t$ is either satisfied from inventory or lost. When $L_t>0$, the cumulative amount satisfied must equal the cumulative available inventory $I_0+q[1,t]$, hence $D[0,t]=L[0,t]+(I_0+q[1,t])\Rightarrow  L[0,t]=D[0,t]-I_0-q[1,t]$. When $L_t=0$, we have $L[0,t]=L[0,t-1]$, and hence in general we find:
\begin{equation}
\label{eq:Lsum}
L[0,t] = \max_{k\in\{0,\ldots,t\}} \left( D[0,k]  - I_0 - q[1,k]  \right)^+.
\end{equation}
(To see this, note that the maximum is attained for the $k$ which corresponds to the period with the last stockout until $t$.) Next for $Y=f(\mathbf{D})$ we have
\begin{equation}
Y = L[0,\tau-1]- \E[L[0,\tau-1]] = \max_{k\in\{0,\ldots,\tau-1\}} \left( D[0,k]-I_0-q[1,k] \right)^+ - \E[L[0,\tau-1]],
\end{equation}
which is clearly non-decreasing in each $D_i$, $i\in\{1,\ldots,\tau\}$. (Note that $Y$ is independent of, and thus non-decreasing in $D_\tau$.) 
Finally observe that
\begin{equation}
\label{eq:Xredef}
X=g(\mathbf{D})=D[0,\tau]-S-\max_{k\in\{0,\ldots,\tau-1\}} \left( D[0,k]-I_0-q[1,k] \right)^+ + \E[L[0,\tau-1]].
\end{equation}
Note that $dX/dD_{\tau} =1$ and $dX/dD_i \in \{0,1\}$ for $i\in\{0,\ldots,\tau-1\}$, thus $g$ is non-decreasing. \Halmos
\end{proof}

\begin{lemma}
\label{lem:LleqBstrong}
If system $\mathcal{B}$ and $\mathcal{L}$ both operate under a PIL policy with
level $U\geq 0$, then $\E[B_t]\geq \E[L_t]$ for any period $t\geq \tau$.
\end{lemma}

\begin{proof}{Proof.}
We will show that $\E[B_{\tau}]\geq \E[L_{\tau}]$ using only that $q_\tau$ can be placed to attain $U\geq0$. By Lemma \ref{lemma:Attainment}
this implies the result.

Let $\tilde{X}$ and $\tilde{Y}$ be two {\em independent} random variables with the same marginal distribution as $X$ and $Y$ respectively.
Observe that
\begin{align}
\E[B_{\tau}] &= \E\left[(X+Y)^+\right] \notag\\
\label{eq:firststep}
& = \int_{z=-\infty}^\infty \P(X\geq z, Y\geq -z) dz \\
\label{eq:esarystep}
& \geq \int_{z=-\infty}^\infty \P(X\geq z) \P(Y\geq -z) dz \\
& = \int_{z=-\infty}^\infty \P(\tilde{X}\geq z) \P(\tilde{Y}\geq -z) dz = \E\left[(\tilde{X}+\tilde{Y})^+\right]\notag,
\end{align}
where \eqref{eq:firststep} follows from Lemma \ref{lem:IvosLemma} and \eqref{eq:esarystep} follows from Theorem 5.1 of \cite{Esaryetal1967}.
Now continuing and using that $\tilde{X}$ and $\tilde{Y}$ are independent, we find 
\begin{align}
 \E\left[(\tilde{X}+\tilde{Y})^+\right] \ge \E\left[(\tilde{X}+\E[\tilde{Y}])^+\right] =  \E[\tilde{X}^+]
 \end{align}
where the inequality follows from Jensen's inequality and the independence between $\tilde{X}$ and $\tilde{Y}$, and the equality holds since $\E[\tilde{Y}]=0$. Note $\E[\tilde{X}^+]=\E[X^+]=\E[L_t]$.  \Halmos
\end{proof}

\begin{lemma}
\label{lemma:costdominate}
If system $\mathcal{B}$ and $\mathcal{L}$ both operate under a PIL policy with
level $U\geq 0$, then for any initial state $\xb$ such that projected inventory $U\geq0$ can be attained, we have $p\E[L_{t}]+h\E[J_{t}]=\E[c_t]\leq \E[c^{\mathcal{B}}_t]=p\E[B_t]+h\E[J_t^\mathcal{B}]$ for any $t\geq\tau$.
\end{lemma}

\begin{proof}{Proof.}
We will show that $\E[c_{\tau}]\leq \E[c^{\mathcal{B}}_{\tau}]$ using only that $q_\tau$ can be placed to attain $U\geq0$. By Lemma \ref{lemma:Attainment}
this implies the result.
The inventory level in $\mathcal{L}$ at the time of arrival of order $q_\tau$ is given by:
\begin{equation}
\label{eq:ItauL}
I_\tau = I_0 + q[1,\tau] - D[0,\tau-1] + L[0,\tau-1].
\end{equation}
Under a PIL-policy with projected inventory $S-\tau\E[D]$, system $\mathcal{L}$ will choose $q_\tau$ such that $\E[I_\tau]=S-\tau\E[D]$.
Using \eqref{eq:ItauL} and solving for $q_\tau$ yields that
\begin{equation}
\label{eq:qtauPIL}
q_\tau=S-I_0- q[1,\tau-1] - \E[L[0,\tau-1]].
\end{equation}
Substituting \eqref{eq:qtauPIL} back into \eqref{eq:ItauL} yields
\begin{equation}
\label{eq:ItauSimplified}
I_\tau = S- D[0,\tau-1] + L[0,\tau-1] - \E[L[0,\tau-1]].
\end{equation}
Now we have for the expected costs that will be incurred in period $\tau$ by system $\mathcal{L}$:
\begin{align}
\E[c_{\tau}] &= h \E[(I_\tau-D_\tau)^+] + p \E[(D_\tau - I_\tau)^+] \notag\\
&= h \E[I_\tau-D_\tau] + h \E[(D_\tau-I_\tau)^+] + p \E[(D_\tau - I_\tau)^+] \notag\\
&= h \E\left[S- D[0,\tau] + L[0,\tau-1] - \E(L[0,\tau-1]) \right] + h \E[L_\tau] + p \E[L_\tau] \notag\\
&= h \E\left[S- D[0,\tau] \right] + h \E[L_\tau] + p \E[L_\tau] \notag\\
&= h\E\left[\left(S- D[0,\tau]\right)^+\right] - h \E\left[\left(  D[0,\tau] - S\right)^+\right] + h \E[L_\tau] + p \E[L_\tau] \notag\\
&\leq h\E\left[\left(S- D[0,\tau]\right)^+\right] + p \E\left[\left(  D[0,\tau] - S\right)^+\right] = \E[c_\tau^\mathcal{B}]=C^{\mathcal{B}*}.
\end{align}
The second and fourth equality follow from using the identity $x=x^+ + (-x)^+$, and the inequality follows from applying Lemma \ref{lem:LleqBstrong} twice. \Halmos
\end{proof}

\begin{theorem}
\label{thm:LdomB}
If system $\B$ is controlled by the optimal base-stock policy, or \emph{equivalently} by a PIL policy with parameter  $S^*-\tau\mu$, and if $\L$ is controlled by a PIL-policy with PIL-level $(S^*-\tau\mu)^+$, then
the cost-rate of system $\L$ dominates the optimal cost-rate of system $\B$, that is 
\[
C(\pil^{U^*})\leq C(\pil^{(S^*-\tau\mu)^+})\leq C^{\mathcal{B}*}.
\]
\end{theorem}

\begin{proof}{Proof.}
Since $\xb_0=\textbf{0}$, we can attain the projected inventory level $U$ in $\L$ in every period. (Regardless of this assumption, the number of periods for which
it is not possible to attain a projected inventory level $U$ in $\L$ is finite almost surely.) Now consider two cases: $S^*\geq \tau\mu$
and $S^*<\tau\mu$.
\begin{enumerate}
    \item Case $S^*\geq\tau\mu$: Without loss of generality let period 0 be the first period in which it is possible to place an order to attain $U=S^*-\tau\mu\geq 0$. By Lemma \ref{lemma:costdominate} it holds that $\E[c_t]\leq\E[c^{\mathcal{B}}_t]=C^{\mathcal{B}*}$ for all $t\geq\tau$ which implies that $C(\pil^{S^*-\tau\mu})\leq C^{\mathcal{B}^*}$. 
    \item Case $S^*<\tau\mu$: Observe first that $C(\pil^0)=p\mu$ because under $U=0$ there is no inventory and all demand is lost. We now have
    $C^{\mathcal{B}*} = C^\mathcal{B}(\pil^{S^*-\tau\mu})
    =p\E[(D[0,\tau]-S^*)^+]+h\E[(S^*-D[0,\tau])^+]\geq
    p\E[D[0,\tau]-S^*]>p\mu=C(\pil^0)$, where the strict inequality holds because $0\leq S^*<\tau\mu$. 
\end{enumerate}
That $C(\pil^{U^*})\leq C(\pil^{(S^*-\tau\mu)^+})$ follows from the definition of $U^*$.\Halmos
\end{proof}

\subsection{Asymptotic optimality as $p\to\infty$}
\label{sec:optimalityLargeP}

To describe penalty cost asymptotics we need the following assumption on the distribution of lead time demand which is identical to
assumption 1 of \cite{Huhetal2009MS} and \cite{Bijvanketal2015}:

\begin{assumption}
\label{assumption:tailnotheavy}
The random variable $D[0,\tau]$ has finite mean and is (i) bounded or
(ii) is unbounded and $\lim_{x\to\infty} \E\big[D[0,\tau]-x \big| D[0,\tau]> x\big]/x=0$.
\end{assumption}

Assumption \ref{assumption:tailnotheavy} is discussed in some detail in Section 3 of \cite{Huhetal2009MS}.
All distributions commonly used to model demand, including Gaussian, gamma, Poisson, negative-binomial, Mixed Erlang,
and Weibull distributions, satisfy this assumption.

\begin{theorem}\label{thm:asymblargeP}
Under assumption 1, the best PIL-policy is asymptotically optimal for the lost sales inventory system as the cost of a lost sale increases:
\[
\lim_{p\to\infty} \frac{C(\pil^{(S^*(p)-\tau\mu)^+})}{C^*}
=\lim_{p\to\infty} \frac{C(\pil^{U^*})}{C^*}=1.
\]
\end{theorem}

\begin{proof}{Proof.}
By Theorem 3 of \cite{Huhetal2009MS} we have 
$\lim_{p\to\infty} C^{\B*}/C^{*}=1$. Combining this with Theorem \ref{thm:LdomB} yields the result. \Halmos
\end{proof}
Recall that by Lemma~\ref{lem:pilisimproving}, a 1-step policy improvement of a constant order policy yields a PIL policy. The former class of policies is not asymptotically optimal for $p\rightarrow \infty$ and tends to perform poorly in this regime, while Theorem~\ref{thm:asymblargeP} demonstrates that the latter class of policies are asymptotically optimal as $p\rightarrow \infty$. Thus policy improvement qualitatively alters asymptotic performance.

\section{Computational aspects}
\label{sec:CompAspects}

In this section we discuss the computation of the projected inventory level $\E[J_{t+\tau-1}|\xb_t]$, the optimization of the projected inventory level $U$, and efficient simulation estimators.  

\subsection{Inventory Projection}\label{sec:InventoryProjection}

Implementation of the PIL policy requires that the projected inventory level, $\E[J_{t+\tau-1}|\xb_t]$, is computed every period $t$. This can be done relatively straightforwardly when demand has a discrete distribution by using the recursive expressions in \eqref{eq:dynamics1} through \eqref{eq:dynamics3}. In this sub-section, we show how $\E[J_{t+\tau-1}|\xb_t]$ can be computed similarly straightforwardly when demand has a Mixed Erlang (ME) distribution. The class of ME distributions is a powerful modeling tool because it can approximate any non-negative distribution arbitrarily closely (cf. Theorem 5.5.1 of \cite{tijms2003}), can be fitted easily on moments, and has many computational advantages in multi-echelon inventory theory \citep{vanhoutum2006}. We next define ME distributions and discuss how to project inventory levels when demand has a ME distribution.

Let $\{E_{i,t}\}_{i=1}^\infty$ be an i.i.d. sequence of exponential random variables with mean $1/\lambda$ for each $t\in\N_0$. Furthermore let $\{K_t\}_{t=0}^\infty$ be a sequence of i.i.d. random variable on the non-negative integers with probability mass function $\P(K_t=k)=\theta_k$, $k\in\N_0$. Demand has an ME distribution when $D_t=\sum_{i=1}^{K_t} E_{i,t}$. A common parameterization of mixed Erlang distributions is based on two moment fitting \citep[e.g.][]{tijms2003,vanhoutum2006}. For convenience we provide this fitting procedure in Section \ref{sec:twomomentfit}.

In order to intuitively explain the efficient inventory projection method that follows, we imbue ME distributed demand with the following interpretation: Each period $K_t$ \emph{customers} arrive, each demanding an exponentially distributed amount of stock. 
Recall that $\E[J_{t+\tau-1}|\xb_t]=I_t+q[t+1,t+\tau-1]-\tau\mu +\E\left[L[t,t+\tau-1] \mid \xb_t\right]$ so that in order to project the inventory level, it suffices to evaluate $\E\left[L[t,t+\tau-1] \mid \xb_t\right]$. The main idea now is to count inventory and lost sales in terms of the number of customers (each with one exponential phase of demand) that can be satisfied. 
Let $\tilde{I}_t$ denote the number of customers whose demand can be met fully with the inventory available at the beginning of period $t$. Then conditional on $\xb_t=(I_t,q_{t+1},q_{t+2},\ldots,q_{t+\tau-1})$, $\tilde{I}_t$ is Poisson distributed with mean $\lambda I_t$. Since $K_t$ customers will arrive in period $t$, there will be $\tilde{L}_t := (K_t-\tilde{I}_t)^+$ customers whose demand cannot be filled completely; note that for one of those customers the demand can be filled partially. Next, we crucially observe that for the customer whose demand is met partially, the amount of demand (in original units) that remains unfulfilled has an exponential distribution with mean $1/\lambda$ due to the lack of memory of the exponential distribution. Thus $\E[L_t \mid \xb_t] = \lambda^{-1} \E[\tilde{L}_t \mid \xb_t]$. Similarly, there is inventory to satisfy the demand of another $\tilde{J}_t:=(\tilde{I}_t-K_t)^+$ customers at the end of period $t$. Let $Q_{t+1}$ have a Poisson distribution with mean $\lambda q_{t+1}$. Then in period $t+1$ there is inventory to satisfy the demand of $\tilde{I}_{t+1}=\tilde{J}_t+Q_{t+1}$ customers. This reasoning can be continued to obtain the following dynamics:
\begin{align}
    \tilde{J}_t &= (\tilde{I}_t-K_t)^+,\\
    \tilde{L}_t &= (K_t - \tilde{I}_t)^+,\\
    \tilde{I}_{t+1}&=\tilde{J}_{t}+Q_{t+1},
\end{align}
with the initial conditions that $\tilde{I}_t$ has a Poisson distribution with mean $\lambda I_t$ and $Q_t$ has a Poisson distribution with mean $\lambda q_t$, where $1/\lambda$ is the mean of the exponential distributions associated with the ME demand. Note that since $\tilde{I}_t$, $\tilde{J}_t$ and $\tilde{L}_t$ all denote a number of customers (whose demand can be met in full or is partially lost) they are distributed on the non-negative integers. Thus the distributions of $\tilde{I}_{t+j}|\xb_t$, $\tilde{J}_t+j|\xb_t$ and $\tilde{L}_{t+j}|\xb_t$ can be computed recursively for $j\in\{0,\ldots,\tau-1\}$; see Appendix \ref{sec:recursions}. 
This recursion will generally require that the distributions of $\tilde{I}_t$, $\tilde{J}_t$, and $Q_t$ are truncated at a sufficiently high level. $K_t$ will usually have a finite support $k_{\max} = \sup\{k\in\N\mid \theta_k>0\}$. In particular, under two moment fits, $K_t$ will have a two-point distribution, see Section \ref{sec:twomomentfit}. Whenever $k_{\max}<\infty$, $\tilde{L}_{t+j}$ will be supported up to $(j+1)k_{\max}$. From this it follows that the distribution of $\tilde{L}_{t+j}$ can be determined when the distributions of $\tilde{I}_t$ and $\tilde{J}_t$ are known up till $(\tau+1) k_{\max}$. Thus, the distributions of $\tilde{I}_t$ and $\tilde{J}_t$ can be truncated at $(\tau+1) k_{\max}$ without affecting the distribution of $\tilde{L}_{t+j}$. This is convenient for computational purposes as there is no need to truncate any infinite summations. Thus we can compute the projected inventory level as
\begin{equation}
\label{eq:projectionefficient}
\E[J_{t+\tau-1} \mid \xb_t]=I_t+q[t+1,t+\tau-1]-\tau\mu + \frac{1}{\lambda}\E\left[\tilde{L}[t,t+\tau-1] \mid \xb_t\right],
\end{equation}
when demand has a ME distribution. The theoretical running time complexity of this procedure is polynomial in $k_{\max}$ and $\tau$. In practice, the procedure is efficient (see Section \ref{sec:LargeTestBed}).

\subsection{Optimization of the projected inventory level policy}
\label{sec:search}
We next discuss the optimization of $U$, i.e., we discuss and study the problem:
\[\min_{U\ge0} C(\pil^U) \]
The expected cost incurred in period $t$ can be written as $c_t =h (I_t-D_t)^+ + p (D_t-I_t)^+= h (I_t- D_t) + (h+p)(D_t-I_t)^+$. Since $\E [I_t]=U$ by Lemma~\ref{lemma:Attainment}, 
we obtain
\begin{equation}\label{eq:costfunctionPIL}
C(\pil^U) = hU-h\E[D]+(h+p)\limsup_{T\rightarrow\infty}  \E\left[\frac{1}{T-\tau+1}L[\tau,T](\pil^U) \right]
\end{equation}
The following theorem wil facilitate the optimization problem. Its proof is in Appendix~\ref{sec:coststructure}.
\begin{lemma}
\label{thm:coststructure}
For any given $\xb_0$ and demand sequence $D_0,\ldots, D_{t-1},D_{t},\ldots,D_{t+\tau}$:
\begin{enumerate}
    \item The cumulative inventory ordered until period $t+\tau$ (i.e. $q[1,t+\tau](\pil^U) | D_0,\ldots,D_{t-1}$) is non-decreasing and concave in $U$.
    \item The cumulative lost demand until period $t+\tau$ (i.e. $L[1,t+\tau](\pil^U) | D_0,\ldots,D_{t+\tau}$) is non-increasing and convex in $U$.
\end{enumerate}
\end{lemma}
The next theorem follows immediately from Lemma \ref{thm:coststructure} and is facilitates optimization of the PIL-level. 
\begin{theorem}
    \label{thm:cpilconvex}
    $C(\pil^U)$ is convex in $U$.
\end{theorem}
\begin{proof}{Proof.}
$\limsup_{T\rightarrow\infty}  \E\left[\frac{1}{T-\tau+1}L[\tau,T](\pil^U) \right]$ is non-increasing and convex in $U$ by Lemma \ref{thm:coststructure}. Combining with \eqref{eq:costfunctionPIL} yields the result. \Halmos
\end{proof}

\subsection{Low variance simulation estimator}
\label{sec:estimator}

The performance of the PIL policy can only be evaluated through simulation. The obvious simulation estimate of $C(\pil^U)$ (or the cost-rate of any other policy) is $\frac{1}{t_2-t_1} \sum_{t=t_1+1}^{t_2} (p L_t + h J_t)$,
where $t_2-t_1$ is a simulation run-length sufficiently long to obtain an accurate estimate and $t_1$ is the warm-up period. 
For the  projected inventory level policy we can exploit results from Sections \ref{sec:search} and \ref{sec:InventoryProjection} to determine an unbiased estimator with much lower variance than this obvious choice. This variance reduction allows us to obtain accurate simulation estimates with much smaller simulation run-times. First we observe from equation \eqref{eq:costfunctionPIL} that another unbiased estimator of $C(\pil^U)$ is given by $\frac{1}{t_2-t_1} \sum_{t=t_1+1}^{t_2} ((p+h) L_t + h(U-\mu))$. Next we observe that in this estimator, $L_t$, can be replaced by $\E[L_{t+\tau-1}\mid \xb_t ]$. This last expectation has a much smaller variance than $L_t$ and its value is computed in each period already as part of projecting the inventory level; see \eqref{eq:projectionefficient}. Thus an unbiased estimator for $C(\pil^U)$ is given by
\begin{equation}
\widehat{C(\pil^U)} = \frac{1}{t_2-t_1} \sum_{t=t_1+1}^{t_2} ((p+h) \E[L_{t+\tau-1}\mid \xb_t ] + h(U-\mu)).
\end{equation}
In numerical computations, we find that the required simulation run-length to obtain a given precision is around two orders of magnitude smaller than the required run-length when using the obvious estimator $\frac{1}{t_2-t_1} \sum_{t=t_1+1}^{t_2} (p L_t + h J_t)$. 

\section{Numerical Results}
\label{sec:NumericalResults}

The PIL policy is asymptotically optimal for large $\tau$ if $D$ has an exponential distribution. In Section \ref{sec:longtaunumerics}, we provide numerical evidence that the PIL also outperforms the COP (and base-stock policy) for non-exponential distributions even for long lead times. Then in Section \ref{sec:zipkintestbed}, we benchmark
the performance of the PIL policy against other policies including
the optimal policy for the standard test-bed of \cite{Zipkin2008numeric}.
Finally in Section \ref{sec:LargeTestBed}, we provide numerical benchmarks for a test-bed of instances of the size one is likely to encounter in practice. All performance evaluations are done with simulation implemented in C. When multiple policies are compared, we use common random numbers. After a warm-up simulation run, run-length and batch-sizes in the simulation are set such that the half-width of a 95\%-confidence interval is less than 1\% of the point estimate, and there is no statistically significant dependence between consecutive batches. This section also provides results on the computational feasibility of computing inventory projections. 

\subsection{Long lead times}
\label{sec:longtaunumerics}

We created a test-bed in which demand has a Mixed Erlang distribution with mean 100 and the holding cost is fixed at $h=1$. We then varied the coefficient of variation of the one period demand ($\sqrt{\Var[D]/(\E[D])^2}$) to be either $\frac{1}{2}$ or $\frac{3}{2}$ using the two moment fitting procedure in Section \ref{sec:twomomentfit}. We also varied the lead-time to be anything between 1 and 20, $\tau\in\{1,\ldots,20\}$ and the penalty cost in the set $p\in\{4,9,19\}$. This makes for a total of $2\cdot3\cdot20=120$ instances. For each of these instances we optimized the constant order policy, the base-stock policy and the PIL policy. The results are shown in Figure \ref{fig:tauplot}. 

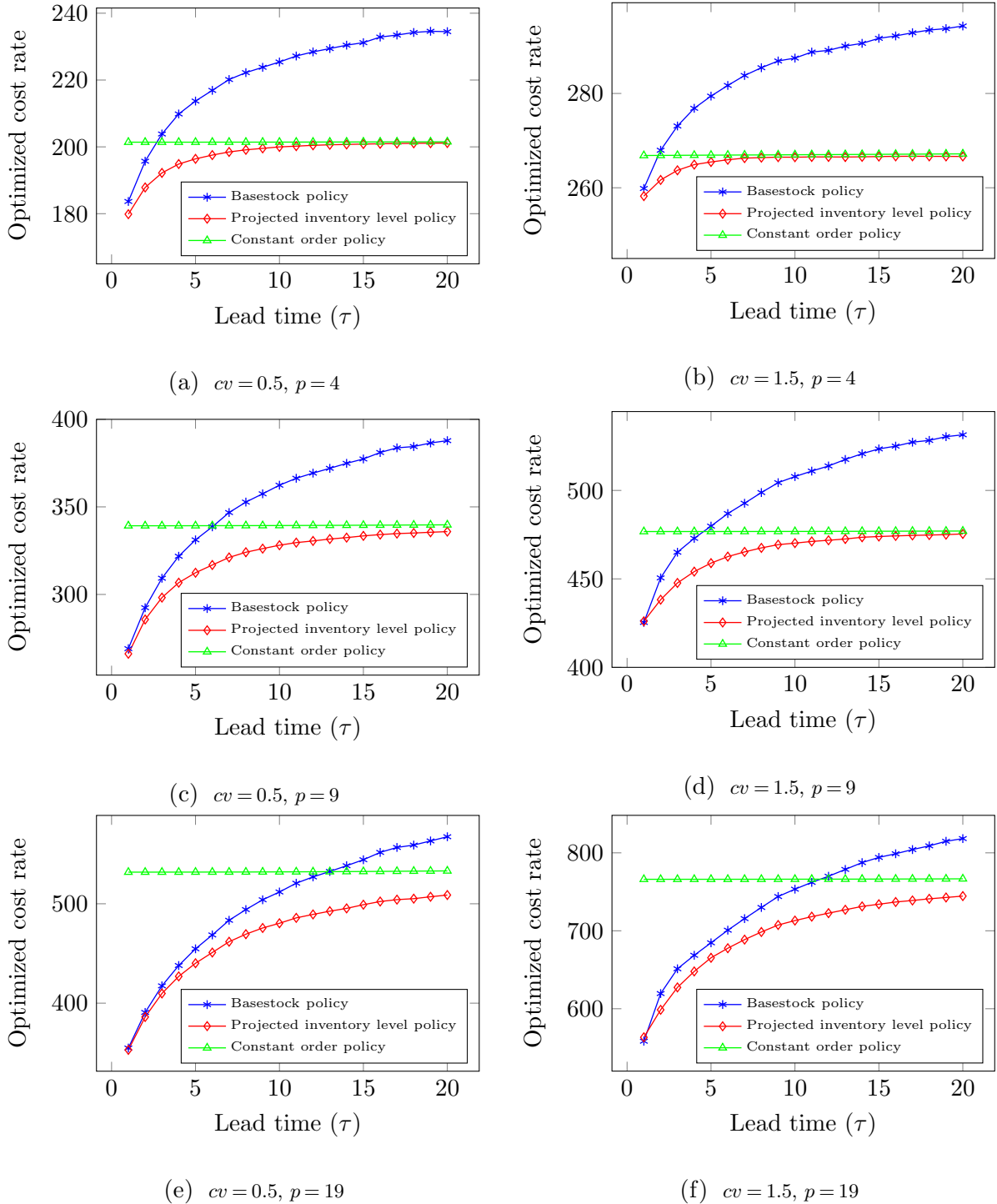
\begin{figure}[h!]
\begin{center}

\begin{subfigure}{0.49\textwidth}
\begin{tikzpicture}
\begin{axis}[xlabel={Lead time ($\tau$)}, ylabel={Optimized cost rate}, width=0.95\textwidth, height=0.7\textwidth, 
legend pos=south east,
legend cell align=left, legend style={font=\tiny}, ymin=165]
\addplot[mark=asterisk, color=blue, line width=.5pt,mark size=2pt] table[x=tau,y=CBS, col sep=comma]{tauplot20220903_MEcvhalf_penalty4.csv};
\addlegendentry{Base-stock policy}
\addplot[mark=diamond, color=red, line width=.5pt,mark size=2pt] table[x=tau,y=CPIL, col sep=comma]{tauplot20220903_MEcvhalf_penalty4.csv};
\addlegendentry{Projected inventory level policy}
\addplot[mark=triangle, color=green, line width=.5pt,mark size=2pt] table[x=tau,y=COP, col sep=comma]{tauplot20220903_MEcvhalf_penalty4.csv};
\addlegendentry{Constant order policy}
legend pos=south east
\end{axis}
\end{tikzpicture}
\caption{\footnotesize $cv=0.5$, $p=4$}
\end{subfigure}
\begin{subfigure}{0.49\textwidth}
\begin{tikzpicture}
\begin{axis}[xlabel={Lead time ($\tau$)}, ylabel={Optimized cost rate}, width=0.95\textwidth, height=0.7\textwidth, 
legend pos=south east,
legend cell align=left, legend style={font=\tiny}, ymin=245]
\addplot[mark=asterisk, color=blue, line width=.5pt,mark size=2pt] table[x=tau,y=CBS, col sep=comma]{tauplot20220903_MEcvonehalf_penalty4.csv};
\addlegendentry{Base-stock policy}
\addplot[mark=diamond, color=red, line width=.5pt,mark size=2pt] table[x=tau,y=CPIL, col sep=comma]{tauplot20220903_MEcvonehalf_penalty4.csv};
\addlegendentry{Projected inventory level policy}
\addplot[mark=triangle, color=green, line width=.5pt,mark size=2pt] table[x=tau,y=COP, col sep=comma]{tauplot20220903_MEcvonehalf_penalty4.csv};
\addlegendentry{Constant order policy}
legend pos=south east
\end{axis}
\end{tikzpicture}
\caption{\footnotesize $cv=1.5$, $p=4$}
\end{subfigure}

\begin{subfigure}{0.49\textwidth}
\begin{tikzpicture}
\begin{axis}[xlabel={Lead time ($\tau$)}, ylabel={Optimized cost rate}, width=0.95\textwidth, height=0.7\textwidth, 
legend pos=south east,
legend cell align=left, legend style={font=\tiny}]
\addplot[mark=asterisk, color=blue, line width=.5pt,mark size=2pt] table[x=tau,y=CBS, col sep=comma]{tauplot20220903_MEcvhalf_penalty9.csv};
\addlegendentry{Base-stock policy}
\addplot[mark=diamond, color=red, line width=.5pt,mark size=2pt] table[x=tau,y=CPIL, col sep=comma]{tauplot20220903_MEcvhalf_penalty9.csv};
\addlegendentry{Projected inventory level policy}
\addplot[mark=triangle, color=green, line width=.5pt,mark size=2pt] table[x=tau,y=COP, col sep=comma]{tauplot20220903_MEcvhalf_penalty9.csv};
\addlegendentry{Constant order policy}
legend pos=south east
\end{axis}
\end{tikzpicture}
\caption{\footnotesize $cv=0.5$, $p=9$}
\end{subfigure}
\begin{subfigure}{0.49\textwidth}
\begin{tikzpicture}
\begin{axis}[xlabel={Lead time ($\tau$)}, ylabel={Optimized cost rate}, width=0.95\textwidth, height=0.7\textwidth, 
legend pos=south east,
legend cell align=left, legend style={font=\tiny}, ymin=400]
\addplot[mark=asterisk, color=blue, line width=.5pt,mark size=2pt] table[x=tau,y=CBS, col sep=comma]{tauplot20220907_MEcvonehalf_penalty9.csv};
\addlegendentry{Base-stock policy}
\addplot[mark=diamond, color=red, line width=.5pt,mark size=2pt] table[x=tau,y=CPIL, col sep=comma]{tauplot20220907_MEcvonehalf_penalty9.csv};
\addlegendentry{Projected inventory level policy}
\addplot[mark=triangle, color=green, line width=.5pt,mark size=2pt] table[x=tau,y=COP, col sep=comma]{tauplot20220907_MEcvonehalf_penalty9.csv};
\addlegendentry{Constant order policy}
legend pos=south east
\end{axis}
\end{tikzpicture}
\caption{\footnotesize $cv=1.5$, $p=9$}
\end{subfigure}

\begin{subfigure}{0.49\textwidth}
\begin{tikzpicture}
\begin{axis}[xlabel={Lead time ($\tau$)}, ylabel={Optimized cost rate}, width=0.95\textwidth, height=0.7\textwidth, 
legend pos=south east,
legend cell align=left, legend style={font=\tiny}]
\addplot[mark=asterisk, color=blue, line width=.5pt,mark size=2pt] table[x=tau,y=CBS, col sep=comma]{tauplot20220903_MEcvhalf_penalty19.csv};
\addlegendentry{Base-stock policy}
\addplot[mark=diamond, color=red, line width=.5pt,mark size=2pt] table[x=tau,y=CPIL, col sep=comma]{tauplot20220903_MEcvhalf_penalty19.csv};
\addlegendentry{Projected inventory level policy}
\addplot[mark=triangle, color=green, line width=.5pt,mark size=2pt] table[x=tau,y=COP, col sep=comma]{tauplot20220903_MEcvhalf_penalty19.csv};
\addlegendentry{Constant order policy}
legend pos=south east
\end{axis}
\end{tikzpicture}
\caption{\footnotesize $cv=0.5$, $p=19$}
\end{subfigure}
\begin{subfigure}{0.49\textwidth}
\begin{tikzpicture}
\begin{axis}[xlabel={Lead time ($\tau$)}, ylabel={Optimized cost rate}, width=0.95\textwidth, height=0.7\textwidth, 
legend pos=south east,
legend cell align=left, legend style={font=\tiny}, ymin=520]
\addplot[mark=asterisk, color=blue, line width=.5pt,mark size=2pt] table[x=tau,y=CBS, col sep=comma]{tauplot20220907_MEcvonehalf_penalty19.csv};
\addlegendentry{Base-stock policy}
\addplot[mark=diamond, color=red, line width=.5pt,mark size=2pt] table[x=tau,y=CPIL, col sep=comma]{tauplot20220907_MEcvonehalf_penalty19.csv};
\addlegendentry{Projected inventory level policy}
\addplot[mark=triangle, color=green, line width=.5pt,mark size=2pt] table[x=tau,y=COP, col sep=comma]{tauplot20220907_MEcvonehalf_penalty19.csv};
\addlegendentry{Constant order policy}
legend pos=south east
\end{axis}
\end{tikzpicture}
\caption{\footnotesize $cv=1.5$, $p=19$}
\end{subfigure}

\end{center}
\caption{Optimized Cost rate as a function of lead time for different policies for Mixed Erlang demand with mean 100 and coefficient of variation ($cv$) of 0.5 (sub-figures a, c, and e) and 1.5 (sub-figures b, d, and f) for penalty costs of 4, 9 and 19 respectively and holding cost rate 1.}
\label{fig:tauplot}
\end{figure}

The PIL policy has superior performance relative to the COP and base-stock policy in all cases. Figure \ref{fig:tauplot} also suggests that the PIL policy is asymptotically optimal as $\tau\to\infty$ for demand distributions other than the exponential distribution.

\subsection{Standard test-bed}
\label{sec:zipkintestbed}

\cite{Zipkin2008numeric} provides a test-bed to compare the performance of notable policies for the canonical lost sales inventory model. This test-bed has relatively small instances as the performance of all policies including the optimal policy are evaluated numerically. 
This test-bed has two demand distributions, Poisson and geometric, both with mean 5. 
The holding cost is fixed at $h=1$ and the other parameters are varied as a full factorial: $\tau\in\{1,2,3,4\}$, $p\in\{4,9,19,39\}$ leading to a total of 32 instances. 
\cite{Zipkin2008numeric} report the performance of notable policies advocated in literature (e.g. \cite{Morton1971,Levietal2008MOOR,HuhJankiraman2011}).
Here we report on the base-stock policy, constant order policy, myopic policy, capped base-stock policy, and the PIL policy. The myopic policy is the best performing policy in Zipkin's test-bed that has intuitive appeal. The myopic policy places an order in period $t$ to minimize the projected cost in period $t+\tau$ given the current state. The myopic policy is defined formally as:
\[
\myo_t(\xb):=\argmin_{q\geq0} \E[pL_{\tau}+hJ_{\tau} \mid \xb_0=\xb].
\]
The capped base-stock policy is formally defined as
\[
\cbs^{S,r}_t(\xb):=\min\{\base^S_t(\xb),\cop^r_t(\xb)\}.
\]
The capped-base-stock policy is also asymptotically optimal both as $p\to\infty$ and as $\tau\to\infty$, as is intuitively clear as the parameters of this policy can be set to mimic either a base-stock policy (by setting the cap $r$ arbitrarily high) or as a constant order policy (by setting the base-stock level $S$ arbitrarily high). 
Table \ref{tab:zipkin} reports the performance of these policies.
The performance of the best capped base-stock policies reported in Table \ref{tab:zipkin} are taken from \cite{xin2019}. He found the optimized parameters by using Matlab's solver ``fmincon''. We also replicated these results but found that whether we found the same policy performance depends on the initial solution provided to the solver. This indicates that the non-convexity of $C(\cbs^{S,r})$ in $S$ and $r$ can be a challenge.

The performance of the PIL policy is closest to optimal with an average optimality gap of 0.6\% whereas the base-stock policy has an average optimality gap of 3.5\%, the myopic policy of 2.8\% and the capped base-stock policy of 0.7\%.
The average performance of the best constant order policy is quite poor with an average optimality gap of 47.4\%. It appears that the PIL policy has attractive asymptotic properties as well as superior empirical performance compared to state of the art heuristics.

\begin{table}[h!]
  \centering
  \scriptsize
  \caption{Comparison of policies on Zipkin's test-bed}
\begin{tabular}{|c|l|cccc|cccc|}
\cline{3-10}\multicolumn{1}{r}{} &       & \multicolumn{4}{c|}{Poisson demand} & \multicolumn{4}{c|}{Geometric demand} \bigstrut\\
\cline{3-10}\multicolumn{1}{r}{} &       & \multicolumn{4}{c|}{Lead-time $\tau$} & \multicolumn{4}{c|}{Lead-time $\tau$} \bigstrut\\
\hline
\multicolumn{1}{|l|}{Penalty per lost sale} & Policy & 1     & 2     & 3     & 4     & 1     & 2     & 3     & 4 \bigstrut\\
\hline
\multirow{6}[2]{*}{$p=4$} & Optimal & \multicolumn{1}{r}{4.04} & \multicolumn{1}{r}{4.40} & \multicolumn{1}{r}{4.60} & \multicolumn{1}{r|}{4.73} & \multicolumn{1}{r}{9.82} & \multicolumn{1}{r}{10.24} & \multicolumn{1}{r}{10.47} & \multicolumn{1}{r|}{10.61} \bigstrut\\
      & PIL   & \multicolumn{1}{r}{4.04} & \multicolumn{1}{r}{4.40} & \multicolumn{1}{r}{4.62} & \multicolumn{1}{r|}{4.74} & \multicolumn{1}{r}{9.84} & \multicolumn{1}{r}{10.28} & \multicolumn{1}{r}{10.51} & \multicolumn{1}{r|}{10.64} \\
      & Myopic & \multicolumn{1}{r}{4.11} & \multicolumn{1}{r}{4.56} & \multicolumn{1}{r}{4.84} & \multicolumn{1}{r|}{5.06} & \multicolumn{1}{r}{9.95} & \multicolumn{1}{r}{10.57} & \multicolumn{1}{r}{10.99} & \multicolumn{1}{r|}{11.31} \\
      & Base-stock & \multicolumn{1}{r}{4.16} & \multicolumn{1}{r}{4.64} & \multicolumn{1}{r}{4.98} & \multicolumn{1}{r|}{5.20} & \multicolumn{1}{r}{10.04} & \multicolumn{1}{r}{10.70} & \multicolumn{1}{r}{11.13} & \multicolumn{1}{r|}{11.44} \\
      & Capped base-stock & \multicolumn{1}{r}{4.06} & \multicolumn{1}{r}{4.41} & \multicolumn{1}{r}{4.63} & \multicolumn{1}{r|}{4.80} & \multicolumn{1}{r}{9.87} & \multicolumn{1}{r}{10.32} & \multicolumn{1}{r}{10.51} & \multicolumn{1}{r|}{10.70} \\
      & COP   & \multicolumn{4}{c|}{5.27}     & \multicolumn{4}{c|}{11.00} \bigstrut\\
\hline
\multirow{6}[2]{*}{$p=9$} & Optimal & \multicolumn{1}{r}{5.44} & \multicolumn{1}{r}{6.09} & \multicolumn{1}{r}{6.53} & \multicolumn{1}{r|}{6.84} & \multicolumn{1}{r}{14.51} & \multicolumn{1}{r}{15.50} & \multicolumn{1}{r}{16.14} & \multicolumn{1}{r|}{16.58} \bigstrut\\
      & PIL   & \multicolumn{1}{r}{5.45} & \multicolumn{1}{r}{6.12} & \multicolumn{1}{r}{6.58} & \multicolumn{1}{r|}{6.90} & \multicolumn{1}{r}{14.55} & \multicolumn{1}{r}{15.60} & \multicolumn{1}{r}{16.27} & \multicolumn{1}{r|}{16.73} \\
      & Myopic & \multicolumn{1}{r}{5.45} & \multicolumn{1}{r}{6.22} & \multicolumn{1}{r}{6.80} & \multicolumn{1}{r|}{7.20} & \multicolumn{1}{r}{14.64} & \multicolumn{1}{r}{15.93} & \multicolumn{1}{r}{16.86} & \multicolumn{1}{r|}{17.61} \\
      & Base-stock & \multicolumn{1}{r}{5.55} & \multicolumn{1}{r}{6.32} & \multicolumn{1}{r}{6.86} & \multicolumn{1}{r|}{7.27} & \multicolumn{1}{r}{14.73} & \multicolumn{1}{r}{15.99} & \multicolumn{1}{r}{16.87} & \multicolumn{1}{r|}{17.54} \\
      & Capped base-stock & \multicolumn{1}{r}{5.48} & \multicolumn{1}{r}{6.12} & \multicolumn{1}{r}{6.62} & \multicolumn{1}{r|}{6.91} & \multicolumn{1}{r}{14.58} & \multicolumn{1}{r}{15.63} & \multicolumn{1}{r}{16.27} & \multicolumn{1}{r|}{16.73} \\
      & COP   & \multicolumn{4}{c|}{10.27}    & \multicolumn{4}{c|}{18.19} \bigstrut\\
\hline
\multirow{6}[2]{*}{$p=19$} & Optimal & \multicolumn{1}{r}{6.68} & \multicolumn{1}{r}{7.66} & \multicolumn{1}{r}{8.36} & \multicolumn{1}{r|}{8.89} & \multicolumn{1}{r}{19.22} & \multicolumn{1}{r}{20.89} & \multicolumn{1}{r}{22.06} & \multicolumn{1}{r|}{22.95} \bigstrut\\
      & PIL   & \multicolumn{1}{r}{6.68} & \multicolumn{1}{r}{7.68} & \multicolumn{1}{r}{8.42} & \multicolumn{1}{r|}{8.95} & \multicolumn{1}{r}{19.28} & \multicolumn{1}{r}{21.03} & \multicolumn{1}{r}{22.73} & \multicolumn{1}{r|}{23.85} \\
      & Myopic & \multicolumn{1}{r}{6.69} & \multicolumn{1}{r}{7.77} & \multicolumn{1}{r}{8.56} & \multicolumn{1}{r|}{9.18} & \multicolumn{1}{r}{19.37} & \multicolumn{1}{r}{21.30} & \multicolumn{1}{r}{22.79} & \multicolumn{1}{r|}{24.02} \\
      & Base-stock & \multicolumn{1}{r}{6.73} & \multicolumn{1}{r}{7.84} & \multicolumn{1}{r}{8.60} & \multicolumn{1}{r|}{9.23} & \multicolumn{1}{r}{19.40} & \multicolumn{1}{r}{21.31} & \multicolumn{1}{r}{22.73} & \multicolumn{1}{r|}{23.85} \\
      & Capped base-stock & \multicolumn{1}{r}{6.69} & \multicolumn{1}{r}{7.72} & \multicolumn{1}{r}{8.40} & \multicolumn{1}{r|}{8.95} & \multicolumn{1}{r}{19.32} & \multicolumn{1}{r}{21.06} & \multicolumn{1}{r}{22.27} & \multicolumn{1}{r|}{23.28} \\
      & COP   & \multicolumn{4}{c|}{15.78}    & \multicolumn{4}{c|}{28.60} \bigstrut\\
\hline
\multirow{6}[2]{*}{$p=39$} & Optimal & \multicolumn{1}{r}{7.84} & \multicolumn{1}{r}{9.11} & \multicolumn{1}{r}{10.04} & \multicolumn{1}{r|}{10.79} & \multicolumn{1}{r}{23.87} & \multicolumn{1}{r}{26.21} & \multicolumn{1}{r}{27.96} & \multicolumn{1}{r|}{29.36} \bigstrut\\
      & PIL   & \multicolumn{1}{r}{7.84} & \multicolumn{1}{r}{9.12} & \multicolumn{1}{r}{10.09} & \multicolumn{1}{r|}{10.91} & \multicolumn{1}{r}{23.94} & \multicolumn{1}{r}{26.37} & \multicolumn{1}{r}{28.18} & \multicolumn{1}{r|}{29.72} \\
      & Myopic & \multicolumn{1}{r}{7.88} & \multicolumn{1}{r}{9.16} & \multicolumn{1}{r}{10.17} & \multicolumn{1}{r|}{11.04} & \multicolumn{1}{r}{23.97} & \multicolumn{1}{r}{26.55} & \multicolumn{1}{r}{28.61} & \multicolumn{1}{r|}{30.31} \\
      & Base-stock & \multicolumn{1}{r}{7.86} & \multicolumn{1}{r}{9.19} & \multicolumn{1}{r}{10.22} & \multicolumn{1}{r|}{11.06} & \multicolumn{1}{r}{24.00} & \multicolumn{1}{r}{26.55} & \multicolumn{1}{r}{28.51} & \multicolumn{1}{r|}{30.12} \\
      & Capped base-stock & \multicolumn{1}{r}{7.84} & \multicolumn{1}{r}{9.14} & \multicolumn{1}{r}{10.08} & \multicolumn{1}{r|}{10.88} & \multicolumn{1}{r}{24.00} & \multicolumn{1}{r}{26.30} & \multicolumn{1}{r}{28.28} & \multicolumn{1}{r|}{29.76} \\
      & COP   & \multicolumn{4}{c|}{18.21}    & \multicolumn{4}{c|}{36.73} \bigstrut\\
\hline
\end{tabular}%

  \label{tab:zipkin}%
\end{table}%

\subsection{Large instance test-bed}
\label{sec:LargeTestBed}

We created a large test-bed of instances for which the optimal 
policy cannot be tractably computed. However, we believe these 
instances give a fair representation of instances that one may 
encounter in practice. In all these instances, demand has a 
Mixed Erlang distribution with $\E[D]=100$ and 
$h=1$. We varied the lead-time $\tau\in\{1,2,3,4,5,6\}$, 
the penalty cost parameter $p\in\{1,4,9,19,49,99\}$, 
and the coefficient of variation of the one period demand $\sqrt{\Var[D]/(\E[D])^2}\in\{0.15,0.25,0.5,1,1.5,2.0\}$ 
for a total of 216 instances. (We use the two moment fitting procedure in Section~\ref{sec:twomomentfit} to fit a ME distribution.)
We use the Matlab solver ``fminbnd'' to optimize the single parameter of the base-stock, and COP for each of these instances. Matlab's multi-dimensional solver ``fmincon'' is used to optimize the parameters of the capped base-stock and the PIL policy with the following parameters that deviate from default: \texttt{FiniteDifferenceStepSize}=$10^{-2}$, \texttt{OptimalityTolerance}=$10^{-3}$, and \texttt{StepTolerance}=$10^{-8}$. These changes to default settings have been made for the benefit of optimizing the capped base-stock policy as the estimator $\widehat{C(\cbs^{S,r})}=\frac{1}{t_2-t_1} \sum_{t=t_1+1}^{t_2} (p L_t + h J_t)$ is not smooth in $S$ and $r$ along a given sample path for the capped base-stock policy. 
Note that ``fmincon" is the solver also used by \cite{xin2019} to optimize the parameters of the capped base-stock policy and represents the state of the art. $C(\cbs^{S,r})$ is not convex in $S$ and $r$, and so there is no guarantee that the best capped stock policy is found, whereas this is guaranteed for the PIL policy due to Theorem \ref{thm:cpilconvex}.

Both the PIL and the capped base-stock policy ran an initial optimization with a smaller simulation run-time ($100\times$ smaller) to provide a good initial solution before running at high simulation precision.
The PIL policy has the best performance of these four policies. Therefore we report the gap of the COP, base-stock and capped base-stock policies relative to the best PIL policy across this test-bed aggregated by setting; see Table \ref{tab:owntestbed}. 
The performance of the PIL policy is again superior to the base-stock and constant order policy by a considerable margin on average (3.64\% and 38.7\%). 
The gap with the capped base-stock policy exceeds the base-stock policy at 3.65\% on average. As the capped base-stock policy includes the base-stock policy as a special case, this indicates that ``fmincon'' cannot always find the best parameters due to the non-convexity of $C(\cbs^{S,r})$. Running the optimization of the capped base-stock policy with multiple initial solutions improves the performance such that the average gap with the best PIL policy drops from 3.65\% to 0.99\%; details on this can be found in Table 2 of \cite{ArtsVanJaarsveldArxivv2}. This approach more than doubles the computational time of optimizing the parameters of the capped base-stock policy.

The simulation of a given PIL policy requires the evaluation of the projected inventory level, $\E[J_{t+\tau-1}\mid \xb_t]$, for each period. Our implementation in C allowed us to do 18 million projections per minute on average for this test-bed. The computation times to optimize policy parameters for different policies are shown in Table \ref{tab:CompSpeed2}. The average time to optimize a PIL policy is slightly less than the time to optimize the capped base-stock policy at 21.41 seconds and 22.91 seconds respectively. Optimization of the capped base-stock policy requires search in two dimensions and long simulation run-times for accuracy. The PIL policy requires projection, but only search in one dimension and shorter simulation run-times due to the low variance estimator in Section~\ref{sec:estimator}.  Both policies have comparable computational requirements for the given test-bed. We do observe that the computation of optimal PIL-levels scales poorly in the lead-times whereas the computation of parameters for the capped base-stock policy remains unchanged as the lead time grows.

\begin{table}[htbp]
\scriptsize
  \centering
  \caption{Comparison of policy performance on large test-bed}
    \begin{tabular}{cc|rrr|rrr|rrr|}
\cline{3-11}    \multicolumn{1}{r}{} &       & \multicolumn{9}{c|}{Percentage gap with PIL policy} \bigstrut\\
\cline{2-11}          & \multicolumn{1}{|l|}{Policy} & \multicolumn{3}{c|}{base-stock} & \multicolumn{3}{c|}{constant order} & \multicolumn{3}{c|}{capped base-stock} \bigstrut\\
\cline{2-11}          & \multicolumn{1}{|l|}{Gap with PIL} & \multicolumn{1}{l}{min} & \multicolumn{1}{l}{max} & \multicolumn{1}{l|}{avg} & \multicolumn{1}{l}{min} & \multicolumn{1}{l}{max} & \multicolumn{1}{l|}{avg} & \multicolumn{1}{l}{min} & \multicolumn{1}{l}{max} & \multicolumn{1}{l|}{avg} \bigstrut\\
    \hline
    \multicolumn{1}{|c|}{\multirow{6}[2]{*}{CV of demand}} & \multicolumn{1}{c|}{0.4} & 0.08  & 16.49 & 4.77  & 0.10  & 199.53 & 49.20 & 0.10  & 16.49 & 4.79 \bigstrut[t]\\
    \multicolumn{1}{|c|}{} & \multicolumn{1}{c|}{0.6} & 0.01  & 14.25 & 4.42  & 0.09  & 176.10 & 44.00 & 0.03  & 14.25 & 4.43 \\
    \multicolumn{1}{|c|}{} & \multicolumn{1}{c|}{0.8} & -0.09 & 11.94 & 4.05  & 0.02  & 160.39 & 40.12 & -0.09 & 11.95 & 4.05 \\
    \multicolumn{1}{|c|}{} & \multicolumn{1}{c|}{1} & -0.15 & 10.17 & 3.49  & -0.04 & 138.85 & 35.00 & -0.14 & 10.17 & 3.49 \\
    \multicolumn{1}{|c|}{} & \multicolumn{1}{c|}{1.2} & -3.81 & 8.61  & 2.83  & 0.02  & 125.56 & 31.70 & -3.80 & 8.61  & 2.83 \\
    \multicolumn{1}{|c|}{} & \multicolumn{1}{c|}{1.4} & -11.23 & 6.71  & 2.31  & -0.02 & 123.77 & 32.01 & -11.23 & 6.72  & 2.32 \bigstrut[b]\\
    \hline
    \multicolumn{1}{|c|}{\multirow{6}[2]{*}{lead time ($\tau$)}} & \multicolumn{1}{c|}{1} & -11.23 & 5.28  & -0.18 & 0.52  & 199.53 & 57.97 & -11.23 & 5.30  & -0.17 \bigstrut[t]\\
    \multicolumn{1}{|c|}{} & \multicolumn{1}{c|}{2} & 0.10  & 9.10  & 2.55  & 0.17  & 158.61 & 46.90 & 0.15  & 9.12  & 2.55 \\
    \multicolumn{1}{|c|}{} & \multicolumn{1}{c|}{3} & 0.24  & 11.80 & 3.60  & 0.07  & 133.20 & 38.72 & 0.26  & 11.80 & 3.61 \\
    \multicolumn{1}{|c|}{} & \multicolumn{1}{c|}{4} & 0.35  & 13.73 & 4.52  & 0.05  & 116.55 & 33.28 & 0.35  & 13.75 & 4.52 \\
    \multicolumn{1}{|c|}{} & \multicolumn{1}{c|}{5} & 0.42  & 15.23 & 5.33  & -0.02 & 104.84 & 29.20 & 0.42  & 15.23 & 5.34 \\
    \multicolumn{1}{|c|}{} & \multicolumn{1}{c|}{6} & 0.83  & 16.49 & 6.05  & -0.04 & 94.97 & 25.95 & 0.84  & 16.49 & 6.06 \bigstrut[b]\\
    \hline
    \multicolumn{1}{|c|}{\multirow{6}[2]{*}{Penalty cost ($p$)}} & \multicolumn{1}{c|}{1} & 1.69  & 16.49 & 7.83  & -0.04 & 4.10  & 0.55  & 1.69  & 16.49 & 7.84 \bigstrut[t]\\
    \multicolumn{1}{|c|}{} & \multicolumn{1}{c|}{4} & 0.02  & 12.90 & 6.58  & 0.26  & 20.75 & 5.25  & 0.02  & 12.90 & 6.58 \\
    \multicolumn{1}{|c|}{} & \multicolumn{1}{c|}{9} & -8.65 & 8.10  & 3.97  & 4.20  & 42.64 & 14.66 & -8.65 & 8.10  & 3.98 \\
    \multicolumn{1}{|c|}{} & \multicolumn{1}{c|}{19} & -11.23 & 4.72  & 2.02  & 14.63 & 74.15 & 31.79 & -11.23 & 4.72  & 2.03 \\
    \multicolumn{1}{|c|}{} & \multicolumn{1}{c|}{49} & -8.79 & 4.82  & 0.88  & 40.57 & 133.93 & 69.04 & -8.79 & 4.82  & 0.89 \\
    \multicolumn{1}{|c|}{} & \multicolumn{1}{c|}{99} & -7.19 & 4.93  & 0.58  & 71.85 & 199.53 & 110.73 & -7.18 & 4.93  & 0.59 \bigstrut[b]\\
    \hline
    \multicolumn{2}{|c|}{Total} & -11.23 & 16.49 & 3.64  & 0.0   & 199.5 & 38.7  & -11.23 & 16.49 & 3.65 \bigstrut\\
    \hline
    \end{tabular}%
  \label{tab:owntestbed}%
\end{table}%

\begin{table}[htbp]
\scriptsize
  \centering
  \caption{Computation times to optimize policy parameters}
    \begin{tabular}{|cc|rrr|rrr|rrr|rrr|}
\cline{3-14}    \multicolumn{1}{r}{} &       & \multicolumn{12}{c|}{Computation times in seconds} \bigstrut\\
\cline{3-14}    \multicolumn{1}{r}{} &       & \multicolumn{3}{c|}{base-stock} & \multicolumn{3}{c|}{constant order} & \multicolumn{3}{c|}{capped base-stock} & \multicolumn{3}{c|}{PIL} \bigstrut\\
\cline{3-14}    \multicolumn{1}{r}{} &       & \multicolumn{1}{l}{min} & \multicolumn{1}{l}{max} & \multicolumn{1}{l|}{avg} & \multicolumn{1}{l}{min} & \multicolumn{1}{l}{max} & \multicolumn{1}{l|}{avg} & \multicolumn{1}{l}{min} & \multicolumn{1}{l}{max} & \multicolumn{1}{l|}{avg} & \multicolumn{1}{l}{min} & \multicolumn{1}{l}{max} & \multicolumn{1}{l|}{avg} \bigstrut\\
    \hline
    \multicolumn{1}{|c|}{\multirow{6}[2]{*}{CV of demand}} & \multicolumn{1}{r|}{0.4} & 7.49  & 12.60 & 9.98  & 8.05  & 12.20 & 10.37 & 3.10  & 56.33 & 22.35 & 1.43  & 118.15 & 34.36 \bigstrut[t]\\
    \multicolumn{1}{|c|}{} & \multicolumn{1}{r|}{0.6} & 6.77  & 11.76 & 9.03  & 6.67  & 10.73 & 8.93  & 2.88  & 56.42 & 20.49 & 1.10  & 69.10 & 21.86 \\
    \multicolumn{1}{|c|}{} & \multicolumn{1}{r|}{0.8} & 6.33  & 12.34 & 9.44  & 6.94  & 10.90 & 8.85  & 13.37 & 99.20 & 25.89 & 0.56  & 97.88 & 16.08 \\
    \multicolumn{1}{|c|}{} & \multicolumn{1}{r|}{1} & 7.04  & 11.62 & 8.93  & 6.68  & 10.33 & 8.27  & 3.08  & 51.90 & 23.48 & 0.50  & 40.19 & 11.02 \\
    \multicolumn{1}{|c|}{} & \multicolumn{1}{r|}{1.2} & 7.63  & 11.86 & 9.19  & 6.04  & 9.83  & 7.71  & 2.90  & 58.36 & 22.68 & 0.59  & 56.85 & 15.32 \\
    \multicolumn{1}{|c|}{} & \multicolumn{1}{r|}{1.4} & 5.89  & 11.16 & 9.27  & 5.15  & 9.88  & 7.18  & 2.97  & 65.91 & 22.58 & 0.63  & 150.78 & 29.83 \bigstrut[b]\\
    \hline
    \multicolumn{1}{|c|}{\multirow{6}[2]{*}{lead time ($\tau$)}} & \multicolumn{1}{r|}{1} & 5.89  & 10.01 & 8.26  & 5.19  & 12.19 & 8.57  & 2.90  & 39.22 & 20.81 & 0.56  & 10.55 & 2.43 \bigstrut[t]\\
    \multicolumn{1}{|c|}{} & \multicolumn{1}{r|}{2} & 7.69  & 11.04 & 9.18  & 5.18  & 12.15 & 8.59  & 2.90  & 56.42 & 25.49 & 0.50  & 25.08 & 6.39 \\
    \multicolumn{1}{|c|}{} & \multicolumn{1}{r|}{3} & 7.69  & 11.86 & 9.30  & 5.15  & 12.15 & 8.62  & 3.13  & 99.20 & 24.35 & 1.06  & 30.00 & 14.25 \\
    \multicolumn{1}{|c|}{} & \multicolumn{1}{r|}{4} & 7.52  & 11.76 & 9.55  & 5.21  & 12.13 & 8.54  & 2.97  & 65.91 & 23.09 & 1.56  & 74.15 & 22.13 \\
    \multicolumn{1}{|c|}{} & \multicolumn{1}{r|}{5} & 6.82  & 12.03 & 9.64  & 5.16  & 11.19 & 8.42  & 3.08  & 41.43 & 22.27 & 2.42  & 125.91 & 34.79 \\
    \multicolumn{1}{|c|}{} & \multicolumn{1}{r|}{6} & 7.67  & 12.60 & 9.91  & 5.16  & 12.20 & 8.56  & 2.88  & 40.14 & 21.46 & 3.53  & 150.78 & 48.49 \bigstrut[b]\\
    \hline
    \multicolumn{1}{|c|}{\multirow{6}[2]{*}{Penalty cost ($p$)}} & \multicolumn{1}{r|}{1} & 6.33  & 11.18 & 8.86  & 5.15  & 8.16  & 6.65  & 11.69 & 43.13 & 23.75 & 0.50  & 11.59 & 3.23 \bigstrut[t]\\
    \multicolumn{1}{|c|}{} & \multicolumn{1}{r|}{4} & 6.77  & 10.21 & 8.51  & 5.76  & 10.00 & 7.77  & 2.90  & 56.42 & 22.00 & 1.20  & 88.42 & 14.70 \\
    \multicolumn{1}{|c|}{} & \multicolumn{1}{r|}{9} & 5.89  & 9.63  & 8.47  & 6.51  & 11.47 & 8.46  & 2.90  & 58.36 & 25.44 & 0.56  & 101.76 & 22.33 \\
    \multicolumn{1}{|c|}{} & \multicolumn{1}{r|}{19} & 7.36  & 10.80 & 9.27  & 7.29  & 11.43 & 8.84  & 3.09  & 99.20 & 21.18 & 0.75  & 125.91 & 27.35 \\
    \multicolumn{1}{|c|}{} & \multicolumn{1}{r|}{49} & 8.39  & 11.76 & 10.32 & 8.14  & 12.19 & 9.64  & 12.82 & 65.91 & 23.54 & 1.98  & 150.78 & 29.47 \\
    \multicolumn{1}{|c|}{} & \multicolumn{1}{r|}{99} & 6.82  & 12.60 & 10.41 & 5.94  & 12.20 & 9.95  & 2.88  & 45.53 & 21.56 & 0.98  & 118.15 & 31.40 \bigstrut[b]\\
    \hline
    \multicolumn{2}{|c|}{Total} & 5.89  & 12.60 & 9.31  & 5.1   & 12.2  & 8.6   & 2.88  & 99.20 & 22.91 & 0.50  & 150.78 & 21.41 \bigstrut\\
    \hline
    \end{tabular}%
  \label{tab:CompSpeed2}%
\end{table}%

\subsection{Discussion}\label{sec:discussion}
A key finding of Sections~\ref{sec:longtaunumerics}-\ref{sec:LargeTestBed} is that the PIL policy consistently outperforms the COP also for non-exponential demand distributions, while the corresponding rigorous proof of Section~\ref{sec:LongLeadTimeRigorous} crucially assumes exponentially distributed demand. We briefly hypothesize in an attempt to understand and partially reconcile the difference in scope.

Recall that we established dominance of the PIL policy over the COP policy for exponential demand, by showing that the PIL policy can be obtained by performing a policy improvement step on the COP policy. For general demand, the $1-$step policy improvement $\pi^{r,+}$ on the constant order policy must satisfy $\pi^{r,+}(\xb_t)\in \argmin_{q\ge 0} \E[\mathcal{H}^r(J_{t+\tau-1}+q)\mid\xb_t]$, where $\mathcal{H}^r(\cdot)$ represents the bias associated with $\cop^r$. For exponentially distributed demand, the bias exists and takes the form of a parabola: As a consequence, $\pi^{r,+}(\xb_t)$ depends only on the expectation of $J_{t+\tau-1}$ and takes the simple form of the PIL policy. However, $\mathcal{H}^r(\cdot)$ will be some general convex function for other demand distributions (provided existence). As a consequence $\pi^{r,+}$ will depend on the precise distribution of $J_{t+\tau-1}$, i.e. it will not be a PIL policy.

However, in line with common approaches in approximate dynamic programming \citep{powell2007approximate}, one might consider to approximate the bias, and in view of the convexity of $\mathcal{H}^r(\cdot)$, it would be reasonable to fit a second-degree polynomial. Such an approximation would yield an \emph{approximate} $\pi^{r,+}$ that is a PIL policy; arguably, our numerical results indicate that such approximations tend to work well. 

\section{Conclusion}
\label{sec:Conclusion}
We introduced the projected inventory level  policy for the lost sales inventory system. We showed that this policy is asymptotically optimal in two regimes and has superior numerical performance. The policy may also be applied to the back-order system, in which case it is equivalent to a base-stock policy.

It needs to be explored whether projected inventory level policies can be fruitfully used for other complicated inventory systems where the whole pipeline is relevant for ordering decisions. Such systems include perishable inventory systems \citep[c.f.][]{bu2020asymptotic}, dual sourcing systems \citep[e.g.][]{xingoldberg2018TBS}, and systems with independent (overtaking) lead-times \citep[e.g.][]{StolyarWang2019}.

\ACKNOWLEDGMENT{\footnotesize The authors are grateful to the department editor, the associate editor, and all the anonymous referees for a thoughtful and constructive review process. The authors thank an anonymous referee for providing the key idea underlying the proof of Lemma~\ref{thm:coststructure}, and Ton de Kok, Ivo Adan and Geert-Jan van Houtum for stimulating discussions.}


\bibliographystyle{ormsv080} 
\bibliography{References} 


\ECSwitch


\ECHead{Technical proofs for companion}

\section{Proof of Lemma \ref{lemma:Attainment}}
\label{sec:ProofAttainment}

\begin{proof}{Proof.}
Suppose actions follow a PIL policy $\pil^U_t(\xb_t)$ for given $U\ge 0$. Suppose the state $\xb_0$ of period 0 satisfies $\E[J_{\tau-1}|\xb_0]\leq U$. Denote by $q_\tau\ge 0 $ the non-negative order placed in period 0 that attains $\E[I_{\tau}|\xb_0]=\E[J_{\tau-1}+q_\tau|\xb_0] =U$. 

Conditional on $D_0$ being $d_0$, the projected inventory level as seen in period 1 before the order $q_{\tau+1}$ is placed is given by $\E[J_{\tau} \mid D_0=d_0]$. It is sufficient to show that 
\begin{equation}
\label{eq:attainstep0}
\forall d_0\ge 0:\E[J_{\tau} \mid \xb_0,D_0=d_0]\leq U,
\end{equation}
since this implies that for period 1 it is always possible to place a non-negative order $q_{\tau+1}$ to attain the projected inventory level $U$. (Periods $2,3,\ldots$ then follow by induction.) Since $J_{\tau}$ is decreasing in $D_0$ for all $D_1,\ldots,D_\tau$ we find $\E[J_{\tau} \mid D_0=d_0]\le \E[J_{\tau} \mid D_0=0]$, and hence to show \eqref{eq:attainstep0} it suffices to demonstrate that $\E[J_{\tau} \mid D_0=0]\leq \E[I_{\tau}]= U$, which will be the subject of the remainder of this proof.

Define the random variable $G(y)$ such that $\P(G(y)\leq g)=\P(I_{\tau}\leq g \mid \xb_t, J_0 = y)$, or equivalently $G(y)= ((\ldots((y+q_1-D_1)^+ + q_2-D_2)^+\ldots)^+ + q_{\tau-1}-D_{\tau-1})^++q_\tau$. Observe that we have
$\E[I_\tau|\xb_0]=\E[G((I_0-D)^+)]$ and $\E[J_\tau\mid D_0 = 0]=\E[(G(I_0)-D)^+]$. Note that $0\leq dG(y)/dy \leq 1$ and $G(y)\geq 0$ so that
\begin{equation}
\label{eq:attainstep1}
(G(I_0)-D)^+ \leq G((I_0-D)^+).
\end{equation}
It follows from \eqref{eq:attainstep1} that $\E[J_\tau \mid D_0=0]=\E[(G(I_0)-D)^+]\leq \E[G((I_0-D)^+)]=\E[I_\tau]$. \Halmos

\end{proof}

\section{Proof of Lemma \ref{lem:bias}}
\label{sec:ProofBias}

\begin{proof}{Proof.}
First note that the inventory at the end of a period under $\cop^r$
satisfies the Lindley recursion $J_{t+1}=(J_t+r-D_t)^+$ so that
$\lim_{t\to\infty}\E[J_t]$ can be determined with the
Pollaczek Khinchine equation for the mean waiting time in an 
$M/D/1$ queue. Therefore we can express $g^r$ in closed form as
\begin{equation}
\label{eq:gr}
g^r = p(\mu-r) + h \frac{r^2}{2(\mu-r)}.
\end{equation}
We can now directly verify that inserting \eqref{eq:biasexplicit} and \eqref{eq:gr}
into the right hand side of \eqref{eq:biasdef} again
equals \eqref{eq:biasexplicit} by using integration by parts and tedious but otherwise straightforward algebra:
\begin{eqnarray}
& &\E_{D}\big[ h(x-D)^+ + p (D-x)^+ + \mathcal{H}^r((x-D)^++r)\big]- g^r \notag\\
& = & p\E\left[(D-x)^+\right] + h \E\left[(x-D)^+\right] + \int_0^x \frac{\mathcal{H}^r(x-y+r)}{\mu}e^{-y/\mu} \diff y + \mathcal{H}^r(r) e^{-x/\mu} - g^r \notag \\
& = & p\mu e^{-x/\mu} + h(x-\mu+\mu e^{-x/\mu}) + e^{-x/\mu}x\left(\frac{h(3r^2+3rx+x^2)}{6\mu(\mu-r)} + 
\frac{pr(2r+x)}{2\mu(\mu-r)} - \frac{p/(2r+x)}{2(\mu-r)}\right) \\
& & +  e^{-x/\mu}  \left(\frac{hr^2}{2(\mu-r)}-pr^2\right)-  p(\mu-r) -h \frac{r^2}{2(\mu-r)} \notag\\
& = & \frac{h}{2(\mu-r)}x^2-px = \mathcal{H}^r(x).
\end{eqnarray}
Clearly $\mathcal{H}^r(x)=\frac{h}{2(\mu-r)}x^2-px$ also satisfies $\mathcal{H}^r(0)=0$.
\Halmos
\end{proof}
\section{Proof of Lemma \ref{lem:Martingaleexpression}}
\label{sec:ProofMartingaleexpression}

\begin{proof}{Proof.}
We prove the following statement by induction on $t_2$, starting at $t_2=t_1$:
\begin{align}
\mathcal{H}^r(I_{t_1})= \E_{D_{t_1},\ldots,D_{t_2}}\left[ c[t_1,t_2](\cop^r) +\mathcal{H}^r(I_{t_2+1}) \middle| I_{t_1}\right]-(t_2+1-t_1)g^r. 
\label{eq:inductionHypothesis}
\end{align}
The base case $t_2=t_1$ holds by \eqref{eq:biasdef}, the definition of $c_{t}$, and $I_{t_1+1}=(I_{t_1}-D_{t_1})+r$. Now for the inductive step, assume that the statement holds for some $t_2\geq t_1$; we will show it holds also for $t_2+1$. Use \eqref{eq:biasdef} to conclude that 
\begin{align}
 \E_{D_{t_1},\ldots,D_{t_2}}\left[ \mathcal{H}^r(I_{t_2+1}) \mid I_{t_1}\right] &= \E_{D_{t_1},\ldots,D_{t_2}}\left[\E_{D_{t_2+1}}[ c_{t_2 +1} + \mathcal{H}^r((I_{t_2+1}-D_{t_2+1})^+ +r) -g^r  | I_{t_2 +1}] \middle| I_{t_1}\right] \notag\\
&=  \E_{D_{t_1},\ldots,D_{t_2}}\left[\E_{D_{t_2+1}}\left[ c_{t_2+1} + \mathcal{H}^r(I_{t_2+2})   \mid I_{t_2+1}\right] \middle| I_{t_1}\right]-g^r\notag\\
\label{eq:onestepbias}
&=  \E_{D_{t_1},\ldots,D_{t_2},D_{t_2+1}}\left[ c_{t_2+1} + \mathcal{H}^r(I_{t_2+2})    \mid I_{t_1} \right]- g^r.
\end{align}
Now substitution of \eqref{eq:onestepbias} back into  the induction hypothesis \eqref{eq:inductionHypothesis} yields:
\begin{align}
\mathcal{H}^r(I_{t_1})&=\E_{D_{t_1},\ldots,D_{t_2}}\left[ c[t_1,t_2](\cop^r) +\mathcal{H}^r(I_{t_2+1}) \mid I_{t_1}\right]-(t_2+1-t_1)g^r \notag\\
&=  \E_{D_{t_1},\ldots,D_{t_2+1}}\left[ c[t_1,t_2+1](\cop^r) +\mathcal{H}^r(I_{t_2+2}) \mid I_{t_1}\right]-(t_2+2-t_1)g^r. \label{eq:martingale} 
\end{align}
By induction, this shows that \eqref{eq:inductionHypothesis} holds for all $t_2\geq t_1$.  \Halmos
\end{proof}

\section{Proof of Lemma \ref{lem:IvosLemma}}
\label{sec:ivosproof}

\begin{proof}{Proof.}
We have
\begin{align}
\label{eq:idlem1}
\E[(X+Y)^+] = \int_{x+y\geq 0} (x+y) dF(x,y)= \int_{x=-\infty}^\infty \int_{y=-x}^\infty (x+y) dF(x,y).
\end{align}
For $x+y\geq 0$ we can write
\begin{equation}
\label{eq:idlem2}
x+y = \int_{z=-y}^x dz.
\end{equation}
Substitution of \eqref{eq:idlem2} into \eqref{eq:idlem1} yields
\begin{align}
\E[(X+Y)^+] &= \int_{x=-\infty}^\infty \int_{y=-x}^\infty \int_{z=-y}^x dz dF(x,y)
= \int_{-y\leq z \leq x} dF(x,y)dz \notag\\
&= \int_{z=-\infty}^\infty \int_{x=z}^\infty \int_{y=-z}^\infty dF(x,y) dz
= \int_{z=-\infty}^\infty \P(X\geq z, Y\geq -z) dz. \mbox{\Halmos}\notag
\end{align}
\end{proof}

  \section{Proof of Lemma~\ref{thm:coststructure}}\label{sec:coststructure}
  \begin{proof}{Proof.}
We first derive an identity for $q[1,T+\tau](P^U)|D_0,\ldots,D_{T-1}$. From (\ref{eq:dynamics1}-\ref{eq:dynamics3}), we find $J_t-L_t=I_t-D_t=J_{t-1}+q_t-D_t$, which implies $J_t-J_{t-1}=L_t+q_t-D_t$. Summing this identity for $t=1,\ldots,T+\tau-1$ yields: $\sum_{t=1}^{T+\tau-1}J_t-J_{t-1}=J_{T+\tau-1}-J_0=L[1,T+\tau-1]+q[1,T+\tau-1]-D[1,T+\tau-1]$. Since $J_0=L_0+I_0-D_0$ by (\ref{eq:dynamics1}-\ref{eq:dynamics2}), we obtain 
\begin{align}
J_{T+\tau-1}=I_0+L[0,T+\tau-1]+q[1,T+\tau-1]-D[0,T+\tau-1]\label{eq:identityforJ}
\end{align} 
Next, note that $\E(J_{T+\tau-1}(\pil^{U})|D_0,\ldots,D_{T-1})=\E(J_{T+\tau-1}(\pil^{U})|\xb_T)$, since $\xb_T$ is known given $D_t,\ldots,D_{T-1}$. Thus, Lemma~\ref{lemma:Attainment} and \eqref{eq:PUt} imply $\E(J_{T+\tau-1}(\pil^{U})|D_0,\ldots,D_{T-1})=U-q_{T+\tau}(\pil^{U})|D_0,\ldots,D_{T-1}$. Taking the conditional expectation on both sides of \eqref{eq:identityforJ} and plugging in this latter identity yields $U-q_{T+\tau}(\pil^{U})|D_0,\ldots,D_{T-1}=I_0+\E(L[0,T+\tau-1]+q[1,T+\tau-1]-D[0,T+\tau-1]|D_0,\ldots,D_{T-1})$. Using $\E(D[0,T+\tau-1]|D_0,\ldots,D_{T-1})=D[0,T-1]+\tau\mu$ and rearranging terms yields:
\begin{align}
&q[1,T+\tau](P^U)|D_0,\ldots,D_{T-1}\nonumber\\&=U-I_0-D[0,T-1]-\tau\mu-\E[L[1,T+\tau-1](P^U)|D_0,\ldots,D_{T-1}]\label{eq:induc_conv_Q}
\end{align} 
We are now ready to prove Properties 1 and 2 by induction. As induction hypothesis, let $T\in \N$ and suppose that $\forall t< T$ and for all demand sequences $D_0,D_1,\ldots$ the random variables $q[1,t+\tau](\pil^U) | D_0,\ldots,D_{t-1}$ are non-decreasing and concave, while the random variables $L[1,t+\tau](\pil^U) | D_0,\ldots,D_{t+\tau}$ are non-increasing and convex. We will prove that these random variables are also non-decreasing and concave (resp. non-increasing and convex) for $t=T$. 

Since $L[1,T+\tau-1](P^U)|D_0,\ldots,D_{T+\tau-1}$ is non-increasing and convex in $U$ by induction hypothesis, and since these properties are preserved in expectation, we find that $\E[(L[1,T+\tau-1](P^U)|D_0,\ldots,D_{T+\tau-1})|D_0,\ldots,D_{T-1}]=\E[L[1,T+\tau-1](P^U)|D_0,\ldots,D_{T-1}]$ is non-increasing and convex in $U$ for every realization $D_0,\ldots,D_{T-1}$. By \eqref{eq:induc_conv_Q} and since $U$ is non-decreasing and concave in $U$, this implies that $q[1,T+\tau](P^U)|D_0,\ldots,D_{T-1}$ is non-decreasing and concave in $U$. Next, specializing \eqref{eq:Lsum} to $P^U$, we find:
\begin{align}
L[0,T+\tau](P^U) & = \max_{k\in \{0,\ldots,T+\tau\}}\big(D[0,k]-I_0-q[1,k](P^U)\big)^+, \label{eq:induc_conv_L} 
\end{align}
Since $L[0,T+\tau](P^U)$ is the maximum of a number of functions, that by induction hypothesis and by the result we just proved are decreasing and convex in $U$, $L[0,T+\tau](P^U)$ must also be decreasing and convex in $U$. This proves the inductive step. 

Now note that our induction hypothesis holds for $T=1$ since in that case the functions are independent of $U$, and hence non-increasing and convex (non-decreasing and concave). This completes our proof by induction of Properties 1 and 2. For Property 3, note that it follows directly from Property 2. \Halmos
\end{proof}

\section{Explicit recursions for inventory projection}
\label{sec:recursions}

All probabilities below will be conditional on the state at time 0, $\xb_0$. We omit this condition as it makes the derivations more readable. 
From equations \eqref{eq:dynamics1}-\eqref{eq:dynamics3}, we obtain the following recursive expressions for any $t\in\{0,\ldots,\tau\}$:
\begin{align}
    \P\left(\tilde{J}_{t}=x\right) &= \sum_{k=0}^{\infty} \P\left((\tilde{I}_t-K_t)^+=x\mid K_t = k\right) \P(K_t=k) = \sum_{k=0}^x \P\left(\tilde{I}_t=x-k\right)\theta_k \\
    \P\left(\tilde{L}_t=x\right) &= \sum_{k=0}^{\infty} \P\left((K_t-\tilde{I}_t)^+=x\mid K_t = k\right) \P(K_t=k) = \sum_{k=x+1}^{\infty} \P\left(\tilde{I}_t=k-x\right) \theta_k\\
    \P\left(\tilde{I}_{t}=x\right) &= \P\left(\tilde{J}_{t-1}+Q_{t}=x\right) = \sum_{y=0}^\infty \P\left(\tilde{J}_{t-1}=x-y\right)\P(Q_{t}=y),
\end{align}
These recursions can be applied running from $t=0$ till $t=\tau$ as $\tilde{I}_0$ and $Q_t$, $t\in\{1,\ldots,\tau-1\}$ have Poisson distributions with means $\lambda I_0$ and $\lambda q_t$ conditional on $\xb_0=(I_0,q_1,\ldots,q_{\tau-1})$ when demand has a ME distribution with scale $\lambda$:
\begin{align}
    \P\left(\tilde{I}_0 = x \right) = e^{-\lambda I_0}\frac{(\lambda I_0)^x}{x!},\qquad
    \P\left( Q_t = x \right) = e^{-\lambda q_t}\frac{(\lambda q_t)^x}{x!}.
\end{align}

\section{Two moment fits for mixed Erlang distributions}
\label{sec:twomomentfit}

This Section provides the two moment fits that can also be found in \cite{tijms2003} or \cite{vanhoutum2006}. Suppose that $D:=\sum_{i=1}^K E_i$ where $\{E_i\}_{i=1}^\infty$ is a sequence of i.i.d. exponential random variables with rate $\lambda$ and that $K$ is an independent discrete random variable on the integers with mass function $\theta_k=\P(K=k)$. 
When the mean $\mu=\E[D]$ of demand and the squared coefficient of variation $c^2_{D}=\Var[D]/\mu^2$ are known, the following two point distribution for $K_t$ and choice for $\lambda$  will match these moments when $c^2_D\leq 1$:
\[
\theta_k = \begin{cases}
p &\text{if $k=k_1$}\\[-12pt]
1-p &\text{if $k=k_2$}\\[-12pt]
0 & \text{otherwise}
\end{cases},
\quad \lambda = \frac{k_1 - p}{\mu},
\]
with
\[
k_1:=\left\lfloor \frac{1}{c^2_D} \right\rfloor,\quad 
k_2=k_1+1,\quad 
p = \frac{k_1 c_D^2-\sqrt{k_1(1+c_D^2)-k_1^2 c_D^2}}{1+c_D^2}.
\]
When $c^2_D>1$, then an appropriate parameterization of a mixed Erlang distribution is
\[
\theta_k = \begin{cases}
p &\text{if $k=k_1$}\\[-12pt]
1-p &\text{if $k=k_2$}\\[-12pt]
0 & \text{otherwise}
\end{cases},\quad
\lambda = \frac{p+k_2 (1-p)}{\mu},
\]
with
\[
k_1=1,\quad 
k_2=\max\left(3,\left\lceil\frac{4c_D^2+\sqrt{16 (c_D^2-1)}}{2}\right\rceil \right), \quad
p = \frac{k_1 c_D^2-\sqrt{k_1(1+c_D^2)-k_1^2 c_D^2}}{1+c_D^2}.
\]

%
%
%




\end{document}